\documentclass[12pt,oneside]{amsart}
\usepackage[utf8]{inputenc}
\usepackage{amsmath,amssymb,latexsym,soul, mathrsfs}
\usepackage{array}
\usepackage{multirow}
\usepackage{multirow}
\usepackage[title,titletoc]{appendix}
\usepackage{color,enumitem,graphicx}
\usepackage[colorlinks=true,urlcolor=blue,
citecolor=red,linkcolor=blue,linktocpage,pdfpagelabels,
bookmarksnumbered,bookmarksopen]{hyperref}
\usepackage[english]{babel}
\usepackage{lineno}
\usepackage[left=2.4cm,right=2.8cm,top=2.5cm,bottom=2.4cm]{geometry}

\usepackage[backend=biber,style=numeric-comp]{biblatex} 
\numberwithin{equation}{section}
\newtheorem{theorem}{Theorem}[section]
\newtheorem{lemma}[theorem]{Lemma}
\newtheorem{conjecture}{Conjecture}
\newtheorem{corollary}[theorem]{Corollary}
\newtheorem{proposition}[theorem]{Proposition}
\newtheorem{cor}[theorem]{Corollary}
\newtheorem{remark}[theorem]{Remark}

\newcommand{\eps}{\varepsilon}
\renewcommand{\epsilon}{\varepsilon}

\renewcommand{\rightarrow}{\to}

\title[Minimizers and best constants]{Minimizers and best constants for a weighted critical Sobolev inequality involving the polyharmonic operator}
\author[J.F.\ de Oliveira]{Jos\'{e} Francisco de Oliveira}\thanks{The first author was partially supported by  CNPq grant numbers 309491/2021-5 and 303443/2025-1}

\author[J.N.\ Silva ]{Jeferson Silva}
\address[J.F.\ de Oliveira]{
\newline\indent Department of Mathematics
	\newline\indent 
	Federal University of Piau\'{i}
	\newline\indent
	64049-550 Teresina, PI, Brazil}
	\email{\href{mailto:jfoliveira@ufpi.edu.br}{jfoliveira@ufpi.edu.br}}
\address[J.N.\ Silva]{\newline\indent Department of Mathematics
	\newline\indent 
	Federal University of the Delta of Parna\'{i}ba
	\newline\indent
	CEP 64202-020, Parna\'{i}ba, PI, Brazil}
	\email{\href{mailto:j.n.silva@ufpi.edu.br}{j.n.silva@ufpi.edu.br}}
\begin{document}
\maketitle
\begin{abstract}
Our main goal is to explicitly compute the best constant for the Sobolev-type inequality involving the polyharmonic operator obtained in (Analysis and Applications 22, pp. 1417–1446, 2024). To achieve this goal, we also establish both regularity and classification results for a generalized critical polyharmonic equation in the radial setting.
\end{abstract}
\section{Introduction and main results}
\noindent For $m,N\in\mathbb{N}$ and $p\ge 1$ with $N>mp$ is well known the Sobolev embedding $\mathcal{D}^{m,p}(\mathbb{R}^{N}) \hookrightarrow L^{p^{*}}(\mathbb{R}^{N})$, where  $p^{*}=pN/(N-mp)$ is the critical exponent, $\mathcal{D}^{m,p}(\mathbb{R}^{N})$ denotes the completion of $C^{\infty}_0(\mathbb{R}^{N})$ in the norm
$
\|u\|=\big(\sum_{|\xi|=m}\|D^{\xi}u\|^{p}_{L^{p}}\big)^{{1}/{p}}
$
and $\|\cdot\|_{L^q}$ denotes the norm in $L^{q}(\mathbb{R}^N)$, $q\ge 1$. The explicit form of the best constant for the above embedding is known at least when $m=1$ or $p=2$, see \cite{Talenti} and \cite{S}. Furthermore, the functions that achieve the best constant are known to be radial and have been completely classified. These extremal functions play a central role in the analysis of elliptic equations involving critical exponents. Such equations have been extensively studied due to their rich structure and applications, see for instance \cite{Brezis-Nirenberg,zbMATH06394199,Gazzola-Grunau-Sweers,Clement-deFigueiredo-Mitidieri} and the references therein.

Recently, the weighted Sobolev spaces including fractional dimensions $W^{m, p}_{T}$, $X^{m, p}_{T}$ and $\mathcal{D}^{m,p}_{T}(\alpha)$ have been studied by several researchers since the pioneering work of P. Mitidieri et al. \cite{Clement-deFigueiredo-Mitidieri}; see, for instance, \cite{Ibero1,Abreu, doOdeOliveira2014, JAJ,DCDS2019,CV2023,DoLuHa,CCM02017}. In this work, we focus on the study of the weighted Sobolev space $\mathcal{D}^{m,p}_{T}(\alpha)$ with $p=2$ and $T=\infty$, which we will briefly describe below. A more-in-depth discussion and references on the topic can be found in Section~\ref{sec2}.

For each non-negative integer $\ell$ and $0< T \leq \infty$, we denote by $AC^{\ell}_{loc}(0,T)$ the set of functions $u:(0,T)\to\mathbb{R}$ that admits the $j$-$th$ derivative $u^{(j)}$  a.e in $(0,T )$ for all $j=0,1, \cdots, \ell$ and
  $u^{(\ell)} \in AC_{loc}(0,T)$, 
where  $AC_{loc}(0,T)$ is the set of all locally absolutely continuous functions in the interval $(0,T)$.  For $\alpha>-1$ and $p\geq 1$, we denote by $L^{p}_{\alpha}=L^{p}_{\alpha}(0,T)$ the weighted Lebesgue space of the measurable functions $u :(0,T)\rightarrow\mathbb{R}$ such that
\begin{equation}\label{norm l_{p} com peso}
\|u\|_{L^{p}_{\alpha}}=\left(\int_{0}^{T}|u(r)|^{p}r^{\alpha}\mathrm{d}r \right)^{\frac{1}{p}}<\infty,
\end{equation}
	which is a Banach space under the standard norm $\| \cdot \|_{L^{p}_{\alpha}} $.

Now,  for any $m\ge 1$ integer number, let $ AC^{m-1}_{\mathrm{R}}(0,T)$ be the set of all functions $u$ belonging to $AC^{m-1}_{loc}(0,T)$ and satisfying the zero right-hand boundary conditions
\begin{equation}\label{testfunctions-right}
 \lim_{r \to T} u^{(j)}(r) = 0,\; \forall j=0,1,\cdots, m-1
\end{equation}
and then we set 
$$D_{0,\infty}(\alpha)=D_{0,\infty}(\alpha,m)=\big\{u\in AC^{m-1}_{\mathrm{R}}(0,\infty)\; :\; u^{(m)}\in L^{2}_{\alpha}(0,\infty)\big\}.$$
For $\alpha-2m+1>0$, we define  the space $\mathcal{D}^{m,2}_{\infty}(\alpha)$ as the completion of $D_{0,\infty}(\alpha)$ under the norm 
\begin{equation}\label{norma-diric}  \|u^{(m)}\|_{L^{2}_{\alpha}} = \left( \int^{\infty}_{0} |u^{(m)}|^{2} r^{\alpha}dr\right)^{\frac{1}{2}}.
\end{equation}
For our purposes, we will introduce a second norm on $\mathcal{D}^{m,2}_{\infty}(\alpha)$ that is strongly related to the generalized radial polyharmonic operator. In fact, we define 
\begin{equation}\label{norma-grad}
    \|u\|_{\nabla^{m}_{\alpha}}= \left(\int_{0}^{\infty}|\nabla^{m}_{\alpha} u|^{2}r^{\alpha}dr\right)^{\frac{1}{2}},
\end{equation}
where
\begin{equation}\label{OpDelta}
   \nabla^{m}_{\alpha} u=\left\{\begin{aligned}
   &\Delta_{\alpha}^{k} u, \;\; &\mbox{if}&\quad m=2k\\
  &\left(\Delta_{\alpha}^{k} u \right)',\;\; & \mbox{if}&\quad m=2k+1 
   \end{aligned}\right. 
\end{equation}
in which
\begin{equation}\label{laplace-geral}
    \Delta_{\alpha} u =r^{-\alpha}(r^{\alpha} u^{\prime})^{\prime}= u''+ \frac{\alpha}{r}u'
\end{equation}
is the $\alpha$-generalized radial Laplacian operator, see \cite{JAJ} for more details. It is worth to mention that the norm \eqref{norma-grad} is associated with the inner product on $\mathcal{D}^{m,2}_{\infty}(\alpha)$ given by
\begin{equation}\label{Inner}
        \langle u, v\rangle=\int_{0}^{\infty} \nabla^{m}_{\alpha} u \nabla^{m}_{\alpha} v r^{\alpha} \,dr, \;\; \mbox{for all}\;\; u, v \in \mathcal{D}^{m,2}_{\infty}(\alpha).
\end{equation}
Furthermore, according with \cite[Proposition~2.15 and Theorem 1.1]{JN-JF} the norms in \eqref{norma-diric}  and \eqref{norma-grad} are equivalent and also we have the continuous embedding 
\begin{equation}\label{D-emb}
    \mathcal{D}^{m,2}_{\infty}(\alpha)\hookrightarrow L^{2^{*}}_{\alpha},\;\;\mbox{with}\;\;\alpha-2m+1>0
\end{equation} 
where 
\begin{equation}\label{critical=exponent}
     2^{*} = 2^{*}(\alpha, m)=\frac{2(\alpha+1)}{\alpha-2m+1}
 \end{equation}
 is the critical exponent. In view of \eqref{D-emb}, for  $ \alpha-2m+1>0$ we can define 
\begin{equation}\label{c0-e40}
  \mathcal{S}=\mathcal{S}(\alpha,m)  = \inf \left\{ \frac{\|\nabla_{\alpha}^{m} u \|^{2}_{L_{\alpha}^{2}}}{\|  u \|^{2}_{L_{\alpha}^{2^{*}}}}\, : \,\, u \in \mathcal{D}^{m,2}_{\infty}(\alpha)\setminus\{0\}\right\}.
\end{equation}
Then, we have $\mathcal{S}>0$ and  $\mathcal{S}^{-\frac{1}{2}}$ is the \textit{best constant} for the Sobolev type embedding \eqref{D-emb}. In \cite[Theorem 1.2]{JN-JF} is proven that the constant $\mathcal{S}(\alpha,m)$ is attained, that is,  there exists $z\in \mathcal{D}^{m,2}_{\infty}(\alpha)$ such that $\|z\|_{L^{2^{*}}_{\alpha}} =1$ and $\|\nabla^{m}_{\alpha} z\|^{2}_{L^{2}_{\alpha}} = \mathcal{S}(\alpha,m).$ 
\subsection{Main results}
The aim of the paper is twofold: first, to identify conditions under which the minimizers of $\mathcal{S}(\alpha,m)$ can be classified; second, to explicitly compute the best constant $\mathcal{S}^{-1/2}(\alpha, m)$.  In this direction, for $m=1$, inspired  by the classic Bliss inequality \cite{Bliss},  P. Cl\'{e}ment et al. \cite[Proposition 1.4]{Clement-deFigueiredo-Mitidieri} were able to show that  $\mathcal{S}(\alpha, 1)$ is attained by the Bliss functions 
\begin{equation}\label{1-bliss function}
 u(r) = \mathcal{P}^{\frac{\alpha-1}{4}}(1+r^{2})^{-\frac{\alpha-1}{2}}, \quad r>0
\end{equation}
where $\mathcal{P}=\mathcal{P}(\alpha,1) = (\alpha-1)\left(\alpha+1\right)$. In addition, we can write
\begin{equation}\label{1-bestC}
\begin{aligned}
  \mathcal{S}^{-\frac{1}{2}}(\alpha,1) &=\mathcal{P}^{-\frac{1}{2}}\left[\dfrac{2\Gamma(\alpha+1)}{\Gamma(\frac{\alpha+1}{2})\Gamma(\frac{\alpha+1}{2})}\right]^{\frac{1}{\alpha+1}},
  \end{aligned}
\end{equation}
where $\Gamma(x)=\int_{0}^{1}(-\ln t)^{x-1}dt,\; x>0$ is the Euler's gamma function. In this paper, we will extend the extremal functions \eqref{1-bliss function} and the best constant \eqref{1-bestC} for $m\ge2$.

From now on, and unless otherwise specified,  we are assuming $\mathcal{D}^{m,2}_{\infty}(\alpha)$, with $\alpha>-1$ and  $m \in \mathbb{N}$ under the \textit{Sobolev condition} (c.f Section~\ref{sec2}) 
\begin{equation}\label{SSS-condi}
    \alpha-2m+1>0.
\end{equation}
Our first result is about the regularity of a weak solution $u\in \mathcal{D}^{m,2}_{\infty}(\alpha)$ to the critical polyharmonic type equation
\begin{equation}\label{problema m}
    (-\Delta_{\alpha})^{m} u = |u|^{2^*-2}u \quad \text{in} \quad (0, \infty),
\end{equation}
where  $2^{*}= 2^{*}(\alpha,m)$ is the critical exponent in \eqref{critical=exponent} and $(-\Delta_{\alpha})^{m}$ denotes the $m$-polyharmonic operators induced by \eqref{laplace-geral}.

\begin{theorem}[Regularity]\label{Existencia de minimizante suave} Suppose $\alpha-2m+1>0$  and 
    let $u \in \mathcal{D}^{m,2}_{\infty}(\alpha)$ be weak solution for
\eqref{problema m}, that is, 
\begin{equation}\label{solution weak}
        \int_{0}^{\infty} \nabla^{m}_{\alpha} u \nabla^{m}_{\alpha} v r^{\alpha} \,dr =  \int_{0}^{\infty} |u|^{2^{*}-2}u v r^{\alpha}\, dr , \quad \forall v \in \mathcal{D}^{m,2}_{\infty}(\alpha).
\end{equation}
Then,  $u\in C^{2m}(0,\infty)$ and solves the equation  \eqref{problema m}. In addition, $u$ satisfies boundary conditions
\begin{equation}\label{0-Ulimites}
\lim_{r\to 0}r^{\alpha}\Delta_{\alpha}^{j} u (r) = \lim_{r\to 0}r^{\alpha}((\Delta_{\alpha})^{j}u)^{\prime} (r)=0,\;\; \mbox{for}\;\; j= 0, 1, \cdots, m-1
\end{equation}
and
\begin{equation}\label{infinty-Ulimites}
\lim_{r\to \infty}\Delta_{\alpha}^{j} u (r) =\lim_{r\to \infty}((\Delta_{\alpha})^{j}u)^{\prime} (r) = 0,\;\; \mbox{for}\;\; j= 0, 1, \cdots, m
\end{equation}
and the identity
\begin{equation}\label{representação deltaj}
    (-\Delta_{\alpha})^{j} u (r) =  \displaystyle\int_{r}^{\infty}\Big( t^{-\alpha} \displaystyle\int_{0}^{t} (-\Delta_{\alpha})^{j+1}u(s) s^{\alpha}\Big)dt,\;\; \mbox{for}\;\; j= 0, 1, \cdots, m-1.
\end{equation}
\end{theorem}
Note that Theorem~\ref{Existencia de minimizante suave} ensures that any weak solution $u$ of the problem \eqref{problema m} is a classical solution. As mentioned before, we intend to calculate the best constant $\mathcal{S}^{-\frac{1}{2}}(\alpha, m)$, for which we must obtain a classification-type result for the solutions of \eqref{problema m}. This is a delicate question since we are considering higher-order derivatives, see for instance \cite{AW2022,Gazzola-Grunau-Sweers}. Here, we say that a function  \( u \in C^{2m}(0,\infty) \) is  $\ell$-th nonsingular at $r=0$ if, for each  $0 \leq i \leq \ell$, there exists
\begin{equation}\label{l-nonsigular}
\lim_{r\to 0^{+}} u^{(i)}(r):= u^{(i)}(0).
\end{equation}
Further, we say that  $u$ is $\ell$-th singular at $r=0$ if $\ell$ is the smallest integer $0 \leq \ell \leq 2m $ such that  
\begin{equation}\label{l-sigular}
\lim_{r\to 0^{+}} |u^{(\ell)}(r)| = +\infty.
\end{equation}

We note that classification results of solutions have been studied by several authors, see \cite{JX1999, FK2019, RS1989, AW2022, CGS1989, WC1991, XL2020, DT2024} and  the references therein. For $N> 2m$ and $p+1 = 2N/(N-2m)$, let
\begin{equation}\label{problema intro}
    (-\Delta)^{m} u = u^{p}, \quad u>0 \quad \text{in}\,\, \mathbb{R}^{N} \setminus\{0\}
\end{equation}
be the classical $m$-polyharmonic equation.
In \cite{JX1999}, it was shown that all $0$-th nonsingular solutions $u$
at the origin $0\in\mathbb{R}^N$ to \eqref{problema intro} are given by
$$u_{\epsilon}(x) = \left( \dfrac{2\epsilon}{\epsilon^{2}+ |x - x_{0}|^2} \right)^{\frac{N - 2m}{2}}, \quad  \epsilon>0\,\, \text{and} \,\, x_{0}\in \mathbb{R}^{N}.$$ On the other hand, the cases $m=1,2, 3$ were analyzed in \cite{FK2019, RS1989, AW2022} and it was established that if $u \in C^{2m}(\mathbb{R}^{N}\setminus\{0\})$ is  $0$-th singular solution to \eqref{problema intro} then $u$ is radially symmetric with respect to the origin, monotonically decreasing and satisfies  $$u (|x|) = |x|^{-\frac{N-2m}{2}} v(- \ln{|x|}),$$ where $v$ is periodic and bounded. For $m\ge 4$,  to the best of our knowledge, the classification result for $0$-th 
 singular solutions to \eqref{problema intro} remains an open question, as noted in \cite[Conjecture~4]{AW2022}. Here, we were able to show a classification-type result for $(2m-1)$-th  nonsingular positive solutions of the generalized $m$-polyharmonic equation \eqref{problema m}. Firstly, for $\epsilon>0$ let us introduce the notation
\begin{equation}\label{c0-e13 +}
 w_{\epsilon}(r) = \mathcal{P}^{\frac{\alpha - 2m +1}{4m}} \left( \dfrac{\epsilon}{\epsilon^{2}+ r^2} \right)^{\frac{\alpha - 2m +1}{2}}, \quad  r>0 
\end{equation}
where $\mathcal{P} = P(\alpha, m)= \displaystyle\prod_{h= -m}^{m-1} (\alpha + 1 + 2h)$. 

We observe that from Theorem~\ref{Existencia de minimizante suave}, the weak solutions $u\in \mathcal{D}^{m,2}_{\infty}(\alpha)$ of \eqref{problema m} are, in fact, of class $C^{2m}(0,\infty)$, which makes it possible to ask whether they are $(2m-1)$-th nonsingular at 
$r=0$ or not. Keeping this in mind, we present the following result:

\begin{theorem}[classification]\label{Teorema unicidade} If $v\in \mathcal{D}^{m,2}_{\infty}(\alpha)\cap C^{2m}(0,\infty) $ is a positive $(2m-1)$-th nonsingular at $r=0$ solution to \eqref{problema m} such that $v^{(i)}(0) = 0$, for  $i = 1, 3, \ldots, 2m-1$, then  $v=w_{\epsilon}$ for some $\epsilon>0$.
\end{theorem}
Theorem~\ref{Teorema unicidade} improves and complements the classification result in \cite[Theorem~1.3]{JX1999}, since we are considering the  non-integer parameter $\alpha$. Here, our approach to prove Theorem~\ref{Teorema unicidade} is different from that in \cite{JX1999}. In fact, the argument in our ODE classification result is closer to that in \cite{S} and \cite{EFJ}.

Before stating the next results, we would like to draw attention to an additional difficulty that we have to deal with the weighted Sobolev spaces $\mathcal{D}^{m,2}_{\infty}(\alpha)$. In the classical case (i.e, $\alpha= N-1$), accordingly to \cite{S} the minimizers $z$ for the variational problem associated with the Sobolev embedding $\mathcal{D}^{m,2}(\mathbb{R}^{N}) \hookrightarrow L^{2^{*}}(\mathbb{R}^{N})$ are $0$-th nonsingular functions. 
Here, it is natural to ask the following question:  \textit{For $m\geq1$, if $z\in \mathcal{D}^{m,2}_{\infty}(\alpha)\cap C^{2m}(0,\infty)$ is a minimizers for $\mathcal{S}(\alpha, m)$,  does it follow that $z$ is a $\ell$-th nonsingular function at $r=0$ for some $0\le \ell\le  2m$?}
For $m=1$, we already know that the Bliss functions in \eqref{1-bliss function} are $0$-th nonsingular at $r=0$ extremal functions for $\mathcal{S}(\alpha,1)$. Furthermore,   for either $\alpha=N-1$ or $m=1$, the answer is also positive. We will show in Lemma~\ref{lemma-iterationRev} below that, for general $\alpha$ and $m=2$, the infimum $\mathcal{S}(\alpha,2)$ admits a $3$rd-nonsingular at $r=0$ minimizer $v$ such that $v^{(1)}(0)=v^{(3)}(0)= 0$, provided that it admits at least one $0$-th nonsingular minimizer. However, unfortunately, we are not able to prove it for $m\ge3$, remaining an interesting open problem. Therefore, we will take into account the following condition:
\begin{equation}\label{cond min non-singular}
    \begin{aligned}
     \mathcal{S}(\alpha, m)&\text{ admits a }(2m-1)\text{-th nonsingular at } r=0 \text{ minimizer }\\ z\in \mathcal{D}^{m,2}_{\infty}(\alpha) 
     &\text{ such that } z^{(i)}(0) = 0, \text{ for }  i = 1, 3, \ldots, 2m-1.  
    \end{aligned}
\end{equation}
It is worth mentioning that the existence of a minimizer $z\in\mathcal{D}^{m,2}_{\infty}(\alpha)$ for $\mathcal{S}(\alpha,m)$ has already been proved in \cite{JN-JF} and, from Theorem~\ref{Existencia de minimizante suave}, we have $z\in C^{2m}(0,\infty)$. Thus, the actual requirements in condition \eqref{cond min non-singular} are merely the nonsingular and odd-order left-zero initial conditions.  

As a by-product of Theorem \ref{Existencia de minimizante suave} and Theorem \ref{Teorema unicidade} we have the following:

\begin{cor}\label{cor 1.2} Suppose that \eqref{cond min non-singular} holds. Then, $z\in \mathcal{D}^{m,2}_{\infty}(\alpha)$ is a minimizer for $\mathcal{S}(\alpha, m)$  if and only if $z =w_{\epsilon}$ (up to a sign) for some $\epsilon>0$.
\end{cor}

Finally, with the help of Corollary \ref{cor 1.2}, we will calculate \textit{best constant}  $\mathcal{S}^{-\frac{1}{2}}(\alpha,m)$ 
for the continuous embedding $\mathcal{D}^{m,2}_{\infty}(\alpha) \hookrightarrow L^{2^{*}}_{\alpha}$. 

\begin{theorem}\label{Theorem 1.3} Suppose that \eqref{cond min non-singular} holds. Then, the best constant $\mathcal{S}^{-\frac{1}{2}}(\alpha, m)$ is given by
$$\mathcal{S}^{-\frac{1}{2}} =  \mathcal{P}^{-\frac{1}{2}}\left[\dfrac{2\Gamma(\alpha+1)}{\Gamma(\frac{\alpha+1}{2})\Gamma(\frac{\alpha+1}{2})}\right]^{\frac{m}{\alpha+1}},$$
    where $\Gamma(x)=\int_{0}^{1}(-\ln t)^{x-1}dt,\; x>0$ is the Euler's gamma function.
\end{theorem}
By combining Lemma~\ref{lemma-iterationRev} below with Theorem~\ref{Theorem 1.3}, we are able to obtain the explicit value for best constant $\mathcal{S}^{-\frac{1}{2}}(\alpha,2)$ under a weaker assumption than \eqref{cond min non-singular}.
\begin{cor} \label{corom=2} If $\mathcal{S}(\alpha,2)$ admits a $0$-th nonsingular minimizer, then  best constant $\mathcal{S}^{-\frac{1}{2}}(\alpha,2)$ is given by
$$\mathcal{S}^{-\frac{1}{2}} =  \mathcal{P}^{-\frac{1}{2}}\left[\dfrac{2\Gamma(\alpha+1)}{\Gamma(\frac{\alpha+1}{2})\Gamma(\frac{\alpha+1}{2})}\right]^{\frac{2}{\alpha+1}}.$$   
\end{cor}
Assuming that $\mathcal{S}(\alpha,m)$ admits at least one $0$-th nonsingular minimizer at $r=0$, according to Theorem~\ref{Theorem 1.3}, as done in  Corollary~\ref{corom=2}, to explicitly compute the best constant $\mathcal{S}^{-\frac{1}{2}}(\alpha,m)$ for 
$m>2$, it is sufficient to prove the following conjecture which we already know to be true in the case $m=2$ (cf. Lemma~\ref{lemma-iterationRev}).

\begin{conjecture}\label{conj} The $0$-th nonsingular minimizers for $\mathcal{S}(\alpha,m)$ are necessarily  $(2m-1)$-th nonsingular at $r=0$ and satisfies the odd-order left-zero initial conditions.
\end{conjecture}
The paper is organized as follows. Section~\ref{sec2} is devoted to a discussion on the weighted Sobolev spaces $W^{m,p}_{T}$, $X^{m,p}_{T}$, and $\mathcal{D}^{m,p}_{T}(\alpha)$, as well as to the proof of the regularity result in Theorem~\ref{Existencia de minimizante suave}. The classification result Theorem~\ref{Teorema unicidade} and Corollary~\ref{cor 1.2} are proven in Section~\ref{sec3}. In Section~\ref{sec4} we show the explicit value of the best constant $\mathcal{S}^{-\frac{1}{2}}$ given by Theorem~\ref{Theorem 1.3} and Corollary~\ref{corom=2}.

\section{Weighted Sobolev spaces, regularity  and  Minimizers behavior}
\label{sec2}
In this section, we provide properties of the weighted Sobolev spaces $W^{m,p}_{T}$, $X^{m,p}_{T}$, and $\mathcal{D}^{m,p}_{T}(\alpha)$ and prove the regularity result in Theorem~\ref{Existencia de minimizante suave}. We also call attention to an estimate for sign-changing solutions of \eqref{problema m} obtained in Lemma~\ref{sinal de minimizante}.

Firstly, for real numbers $\alpha_{j}> -1$ with $j=0,1, \cdots, m$, we set
$$W^{m,p}_{T}=W^{m,p}_{T}(\alpha_{0}, \cdots, \alpha_{m}) = \{u\in AC^{m-1}_{loc}(0,T)\, :\,\, u^{(j)} \in L^{p}_{\alpha_{j}}, \quad j=0, 1, \cdots, m\}, $$
where $0<T\le \infty$. We recall  that $W^{m,p}_{T}$ is a Banach space endowed with the norm
\begin{equation}\label{c0-norma1}
    \|u\|_{W^{m,p}_{T}} = \left( \sum_{j=0}^{m} \| u^{(j)}\|_{L^{p}_{\alpha_{j}}}^{p} \right)^{\frac{1}{p}}.
\end{equation}
Alternatively, we can define $W^{m,p}_{T}$ as the subspace of the functions $u\in L^{1}_{loc}(0,T)$ which have the $j$-$th$ distributional
derivative $u^{(j)}\in L_{\alpha_{j}}^{p}$,  for all $j=0,1, \cdots, m$, see \cite[Proposition A]{DoLuRa2}. For the unbounded case $T=\infty$ and assuming the transition condition
\begin{equation}\label{trasition-C}
  \alpha_{j-1}\ge \alpha_{j}-p,\quad j=1,\cdots, m
\end{equation} 
the authors in \cite[Lemma~2.1]{DoLuRa2} were able to  provide a third characterization for $W^{m,p}_{\infty}$ as the closure of the set $\Upsilon=\{ u_{|_{[0, \infty)}} \,: \, u \in C^{\infty}_{c}(\mathbb{R}) \}$ under the norm \eqref{c0-norma1}, where $C^{\infty}_{c}(\mathbb{R})$ is the space of real smooth functions with  compact support in $\mathbb{R}$.

On the other hand, we define $X^{m, p}_{T} = X^{m,p}_{T}(\alpha_{0}, \alpha_{1}, \cdots, \alpha_{m})$ as the closure of the set $$W_{0,T}=AC^{m-1}_{\mathrm{R}}(0,T)\cap W^{m,p}_{T}(\alpha_{0}, \cdots, \alpha_{m})$$ under the norm \eqref{c0-norma1}.

We have  $X^{m,p}_{T} \subset W^{m,p}_{T}$, $0<T\leq \infty$. Additionally, if $T= \infty$ and the weights satisfy \eqref{trasition-C}, it follows from \cite[Lemma~2.1]{DoLuRa2} that both $X^{m,p}_{\infty}$ and $W^{m,p}_{\infty}$ are the closure of $\Upsilon$ under the norm \eqref{c0-norma1}, and thus $X^{m,p}_{\infty}=W^{m,p}_{\infty}$. On the other hand, if $0<T<\infty$ and \eqref{trasition-C} hold, the expression
\begin{equation}\label{norma-diricNOVA}  \|u^{(m)}\|_{L^{p}_{\alpha_{m}}} = \left( \int^{T}_{0} |u^{(m)}|^{p} r^{\alpha_{m}}dr\right)^{\frac{1}{p}}
\end{equation}
defines a norm on $X^{m,p}_{T}(\alpha_0,\cdots, \alpha_m)$ which is equivalent to the norm \eqref{c0-norma1}. Also, under these assumptions, we can consider a third  equivalent norm on  $X^{m,p}_{T}$ given by     
\begin{equation}\label{norma-gradNOVA}
    \|u\|_{\nabla^{m}_{\alpha_{m}}}= \left(\int_{0}^{T}|\nabla^{m}_{\alpha_{m}} u|^{p}r^{\alpha_m}dr\right)^{\frac{1}{p}},
\end{equation}
where  the operator $\nabla^{m}_{\alpha_{m}}$ is given in \eqref{OpDelta}.

We note that the behavior of a function $u\in AC^{m-1}_{loc}(0,T)$ that is admitted in the space $X^{m,p}_{T}(\alpha_0,\cdots, \alpha_m)$ is influenced by the parameters  $\alpha_0,\cdots, \alpha_m,m$ and $p$. However, this dependence is well structured.  In fact, we can distinguish  three main cases:

\paragraph*{I}\textit{Sobolev: $\alpha_m-mp+1>0$.} In this case, under the condition 
$\theta\ge \alpha_{m}-mp$ we have the continuous embedding
	
 \begin{equation}\label{e7full X}
      X^{m,p}_T (\alpha_0,\cdots, \alpha_m) \hookrightarrow L^{q}_{\theta} \quad \text{if}\,\,p\leq q \leq p^{*}, 
  \end{equation}
 where  
 \begin{equation}
     p^{*} = p^{*}(m, p, \theta, \alpha_{m})=\frac{(\theta+1)p}{\alpha_m-mp+1}
 \end{equation}
 is the critical exponent. Moreover, the embedding is compact if $p\le q<p^{*}$.  Here,  if $T=\infty$ we are assuming that $\alpha_{j} \geq \alpha_{m}-(m-j)p$, for $j=0, 1, \cdots, m-1$ and $\alpha_{0} \ge \theta$. Also, the compact embedding is ensured for $q=p$ only if $\alpha_0>\theta$.  See \cite{DoLuHa, DoLuRa2,JAJ} for more details. 
 
\paragraph*{II}\textit{Trudinger-Moser:  $\alpha_m-mp+1=0.$} Now, under the condition 
$\theta>-1$, we have the compact embedding	
 \begin{equation}\label{TM-Xemb}
       X^{m,p}_T(\alpha_0,\cdots, \alpha_m) \hookrightarrow L^{q}_{\theta} \quad \text{if}\,\,p\leq q <\infty.
  \end{equation}
If $T=\infty$ we are supposing $\alpha_{j} \geq \alpha_{m}-(m-j)p$, for $j=0, 1, \cdots, m-1$ and $\alpha_{0} \ge \theta$. Also, the embedding is proved for $q=p$ only if $\alpha_0>\theta$. More precisely, the	maximal growth for a function $\phi(s), s\in\mathbb{R}$ such that $\phi\circ u\in L^{1}_{\theta}$,  for all $u\in  X^{m,p}_T$ is of the exponential-type, see \cite{DoLuHa, CCM02017, JAJ, doOdeOliveira2014} for $0<T<\infty$ and \cite{DoLuRa3, Ibero1, doOdeOliveira2014,JFO} for $T=\infty$.

\paragraph*{III}\textit{Morrey:  $\alpha_m-mp+1<0$.} In this case, for  $0<T<\infty$ and $\alpha_m\ge 0$,  the authors in \cite{DoLuHa} were able to show the continuous embedding 
 \begin{equation}\label{M-embedding}
    X^{m,p}_T(\alpha_0,\cdots, \alpha_m)\hookrightarrow  C^{m-\lfloor \frac{\alpha_m+1}{p}\rfloor -1,\gamma}([0,T]),
 \end{equation}
where $\gamma\in (0,1)$ if $\frac{\alpha_m+1}{p}\in \mathbb{Z}$ and 
\begin{equation}\nonumber
    \gamma=\min\Big\{1+\Big\lfloor \frac{\alpha_m+1}{p}\Big\rfloor-\frac{\alpha_m+1}{p}, 1-\frac{1}{p}\Big\}, \;\;\mbox{if}\;\; \frac{\alpha_m+1}{p}\not\in \mathbb{Z}.
\end{equation}
Analogous results for cases $\mathrm{I}$, $\mathrm{II}$ and $\mathrm{III}$ are proven by \cite{JFO,DoLuHa,DoLuRa2} for the weighted Sobolev spaces without zero boundary conditions $W^{m,p}_T$.

We recall that the norm \eqref{norma-gradNOVA} is related to  operator in \eqref{laplace-geral}, but for $T=\infty$ neither \eqref{norma-gradNOVA} nor \eqref{norma-diricNOVA} define norm on $ X^{m,p}_{\infty}(\alpha_0,\cdots, \alpha_m)$ except for the specific choices $\alpha_{\ell}=\alpha_m-(m-\ell)p$, see \cite[Corollary 2.5]{JN-JF}. In order to overcome this,  in \cite{JN-JF} the authors introduced the weighted Sobolev space $\mathcal{D}^{m,p}_{T}(\alpha)$ as follows: For $p\geq 1$, $\alpha>-1$, $0<T\le \infty$  and $m \in \mathbb{N}$  we set 
$$D_{0,T}(\alpha)=D_{0,T}(\alpha,m,p)=\big\{u\in AC^{m-1}_{\mathrm{R}}(0,T)\; :\; u^{(m)}\in L^{p}_{\alpha}(0,T)\big\}.$$ 
Then, $\mathcal{D}^{m,p}_{T}(\alpha)$ is the completion of $D_{0,T}(\alpha,m,p)$ under the norm  \eqref{norma-diricNOVA} with $\alpha_m=\alpha$. Under the Sobolev condition $\alpha-mp+1>0$, the expression \eqref{norma-gradNOVA}  also defines norm on $\mathcal{D}^{m,p}_{T}(\alpha)$ which is equivalent to \eqref{norma-diricNOVA}, see \cite[Proposition 2.15]{JN-JF}. In addition, we have $X^{m,p}_{\infty}(\alpha_0,\cdots, \alpha_m)=\mathcal{D}^{m,p}_{\infty}(\alpha_m)$ for the suitable choices $\alpha_{\ell}=\alpha_m-(m-\ell)p$, see \cite[Corollary 2.5]{JN-JF}. In general, we have inclusion 
$X^{m,p}_{\infty}(\alpha_0,\cdots, \alpha_m)\subset\mathcal{D}^{m,p}_{\infty}(\alpha_m)$, but the reverse inclusion may not hold true. For example, the function $u(r) = (1+r^2)^{-\alpha/4}$ with $\alpha>1$ belongs to  $\mathcal{D}^{1,2}_{\infty}(\alpha)$ but it does not belong to $X^{1, 2}_{\infty}(\alpha-1, \alpha)$. 

We observe that \cite[Theorem 1.1]{JN-JF} provides the continuous embedding 
\begin{equation}\label{D-embNOVO}
    \mathcal{D}^{m,p}_{T}(\alpha)\hookrightarrow L^{p^{*}}_{\theta},\;\;\mbox{with}\;\; \theta\ge \alpha-mp\;\;\mbox{and}\;\;\alpha-mp+1>0
\end{equation} 
which allow us  to consider the general minimizer problem
\begin{equation}\label{c0-e40NOVO}
  \mathcal{S}=\mathcal{S}(\alpha,m,p,\theta, T)  = \inf \left\{ \frac{\|\nabla_{\alpha}^{m} u \|^{p}_{L_{\alpha}^{p}}}{\|  u \|^{p}_{L_{\theta}^{p^{*}}}}\, : \,\, u \in \mathcal{D}^{m,p}_{T}(\alpha)\setminus\{0\}\right\}.
\end{equation}
Problem~\ref{c0-e40NOVO} was recently investigated in \cite{JN-JF}.

Lastly, we present below a table showing the relations among the spaces $X^{m,p}_T(\alpha_0,\cdots, \alpha_m)$, $W^{m,p}_T(\alpha_0,\cdots, \alpha_m)$ and $\mathcal{D}^{m,p}_{T}(\alpha_m)$.
\begin{table}[h!]
\centering
\begin{tabular}{| m{4cm} | m{7cm} | m{4cm} |} 
 \hline
 \multicolumn{1}{|c|}{$ T>0 $} &  \multicolumn{1}{|c|}{$\alpha_0,\cdots, \alpha_m$} &  \multicolumn{1}{|c|}{Relationship}\\  
 \hline
 \multirow{1}{4em}{arbitrary} & \multirow{1}{4em}{arbitrary} & $X^{m,p}_{T} \subset \mathcal{D}^{m,p}_{T}\cap W^{m,p}_{T}$\\ 
 \hline 
 $T<\infty$ & $\alpha_{j-1}\geq \alpha_{j}-p$ & $X^{m,p}_{T} = \mathcal{D}^{m,p}_{T}$\\  
 \hline  
 \multirow{2}{4em}{$T=\infty$} & $\alpha_{j-1}\geq \alpha_{j}-p$ & $W^{m,p}_{\infty} = X^{m,p}_{\infty} \subset \mathcal{D}^{m,p}_{\infty}$\\
 & $\alpha_{j} = \alpha_m - (m-j)p$ and $\alpha_m - mp + 1 > 0$ & $W^{m,p}_{\infty} = X^{m,p}_{\infty} = \mathcal{D}^{m,p}_{\infty}$\\
 \hline
\end{tabular}
\end{table}

We now prove a pointwise estimate along the same lines as the radial Lemma in \cite{Strauss}.
\begin{lemma}\label{LR2}
    Let $u \in \mathcal{D}^{m,p}_{\infty}(\alpha)$ with $p\geq 2$. Assume that $\alpha-mp+1>0$, then
    \begin{equation}\label{PEradial}
        |u(r)|\leq c r^{-\frac{(\alpha-mp+1)}{p}} \|u\|_{\nabla_{\alpha}^{m}}, \quad \forall\, r>0
    \end{equation}
for some $c>0$ (independent of $u$).
\end{lemma}
\begin{proof} From  \cite[Corollary 2.5]{JN-JF}, we have $\mathcal{D}^{m,p}_{\infty}(\alpha) = X^{m,p}_{\infty}(\beta_{0}, \cdots, \beta_{m})$  and
\begin{equation}\label{TTT}
      \begin{aligned}
\|u^{(j)}\|_{L^{p}_{\beta_j}} \le  C \|u^{(m)}\|_{L^{p}_{\alpha}}, \;\;\mbox{for all}\;\;  u \in \mathcal{D}^{m,p}_{\infty}(\alpha)
     \end{aligned}
\end{equation}
where  $\beta_{j}=\alpha-(m-j)p$. Now, for $u\in X^{m,p}_{\infty}(\beta_{0}, \cdots, \beta_{m})\subset X^{1,p}_{\infty}(\beta_{0}, \beta_{1})$, \cite[Lemma 2.3]{JFO} ensures
\begin{equation}\label{RTTT}
    |u(r)|\leq C_{1} r^{-\gamma}\Big(\|u^{(0)}\|_{L^{p}_{\beta_0}}+\|u^{(1)}\|_{L^{p}_{\beta_1}}\Big)^{\frac{1}{p}}, \quad \forall\, r>0
\end{equation}
where $$\gamma=\frac{\beta_1+(p-1)\beta_0}{p^2}=\frac{\alpha-mp+1}{p}.$$
By combining  \eqref{TTT}, \eqref{RTTT} and \cite[Proposition 2.15]{JN-JF} we conclude that \eqref{PEradial} holds.
\end{proof}
\begin{lemma}\label{lemma crucial D} Assume $p\ge 1$,  $\alpha>-1$ and  $m \in \mathbb{N}$ with $\alpha-mp+1>0.$ Then, the set $\Upsilon=\{ u_{|_{[0, \infty)}} \,: \, u \in C^{\infty}_{c}(\mathbb{R}) \}$ is dense in $\mathcal{D}^{m,p}_{\infty}(\alpha)$.
\end{lemma}
\begin{proof}
From \cite[Corollary 2.5, Remark 2.6]{JN-JF} we have $\mathcal{D}_{\infty}^{m,p}(\alpha)=X_{\infty}^{m,p}(\alpha_0, \cdots,\alpha_m)$, for the choices $\alpha_{j} = \alpha - (m-j)p$ with $j=0,1, \cdots, m$. 
Then, the result follows from \cite[Lemma~2.1]{DoLuRa2}.
\end{proof}

The next result plays an important role in the proof of the Theorem~\ref{Existencia de minimizante suave}.
\begin{lemma} \label{lemma-iteration} Let $u \in \mathcal{D}^{m,2}_{\infty}(\alpha)$ and $1\le k\le m$. Define $w_k:(0,\infty)\to \mathbb{R}$  given by
\begin{equation}\label{wk-iteration}
 w_{0}=|u|^{2^{*}-2}u \;\;\mbox{and}\;\;   w_{k}(r) = \int_{r}^{\infty} t^{-\alpha}\left( \int_{0}^{t}s^{\alpha} w_{k-1}(s)  \,ds\right)dt \quad (k\ge 1)
\end{equation}
Then, for $r>0$, the following properties hold:
\begin{align}
     &|w_{k}(r)| \leq C r^{- \frac{\alpha+2m+1-4k}{2}}\|u\|^{2^{*}-1}_{\nabla^{m}_{\alpha}}\label{wk-estimate}\\
     & |w'_{k}(r)| \leq C r^{- \frac{\alpha+2m+3-4k}{2}}\|u\|^{2^{*}-1}_{\nabla^{m}_{\alpha}}\label{w'k-estimate}\\
     & w_{k}\in C^{2k}(0,\infty) \label{wkc2k}\\
      & (-\Delta_{\alpha})^{j}w_{k}=w_{k-j}, \;\;  1\le j\le k. \label{Dita}
\end{align}
\end{lemma}
\begin{proof}
    To prove \eqref{wk-estimate} and \eqref{w'k-estimate}, we will proceed by induction. If $k=1$, the H\"{o}lder inequality and the embedding $\mathcal{D}^{m,2}_{\infty}(\alpha) \hookrightarrow L^{2^{*}}_{\alpha}$ yield
 \begin{equation}\label{afirmacao 1 est}
     \left|\int_{0}^{t} w_{0}(s) s^{\alpha}\, ds\right| \leq \|u\|^{2^{*}-1}_{L^{2^{*}}_{\alpha}}\left(\dfrac{t^{\alpha+1}}{\alpha+1}\right)^{\frac{1}{2^{*}}}\le Ct^{\frac{\alpha-2m+1}{2}}\|u\|^{2^{*}-1}_{\nabla^{m}_{\alpha}}.
 \end{equation}
Consequently, there is  $C_1=C_1(\alpha, m)>0$  such that 
$$|w_{1}(r)| \leq C\|u\|^{2^{*}-1}_{\nabla^{m}_{\alpha}} \int_{r}^{\infty} t^{-\frac{\alpha+2m-1}{2}}\, dt = C_1r^{- \frac{\alpha+2m-3}{2}}\|u\|^{2^{*}-1}_{\nabla^{m}_{\alpha}},$$
which represents \eqref{wk-estimate} for $k=1$. Assume  that \eqref{wk-estimate} holds for $1\leq k<m$. Thus, since $\alpha-2m+1+4k\ge \alpha-2m+1>0$ we have 
\begin{equation*}
    \left|\int_{0}^{t} w_{k}(s) s^{\alpha}\, ds\right| \le
     C\|u\|^{2^{*}-1}_{\nabla^{m}_{\alpha}}\int_{0}^{t} s^{\frac{\alpha-2m-1+4k}{2}}\, ds
    =Ct ^{\frac{\alpha-2m+1+4k}{2}}\|u\|^{2^{*}-1}_{\nabla^{m}_{\alpha}}.
\end{equation*}
By using $\alpha+2m+1=\alpha-2m+1+4m> 4m\ge 4(k+1)$ we can integrate again to obtain
\begin{eqnarray*}
    |w_{k+1}(r)| &=& \left|\int_{r}^{\infty}t^{-\alpha} \int_{0}^{t} w_{k}(s) s^{\alpha}\, ds \right|\\
&\leq& C\|u\|^{2^{*}-1}_{\nabla^{m}_{\alpha}} \int_{r}^{\infty} t^{-\frac{\alpha+2m-1-4k}{2}}\, dt \\
&=& C r^{- \frac{\alpha+2m+1-4(k+1)}{2}}\|u\|^{2^{*}-1}_{\nabla^{m}_{\alpha}}.
\end{eqnarray*}
So, \eqref{wk-estimate} holds for $k+1$. Next, we will prove \eqref{w'k-estimate}. Firstly, from  \eqref{wk-estimate} each  $w_k$ is well-defined and we can write
\begin{equation}\label{w'k-integral}
w_{k}'(r) = -r^{-\alpha} \int_{0}^{r} w_{k-1}(s) s^{\alpha} \,ds,\;\;\mbox{for}\;\; 1\le k\le m.
\end{equation}
Thus, with help of \eqref{afirmacao 1 est} it follows that 
\begin{eqnarray*}
    |w_{1}'(r)| =  r^{-\alpha} \left| \int_{0}^{r} w_{0}(s) s^{\alpha} \,ds\right|
    \leq C r^{-\alpha} r^{\frac{\alpha-2m+1}{2}}\|u\|^{2^{*}-1}_{\nabla^{m}_{\alpha}} 
    = Cr^{-\frac{\alpha+2m-1}{2}}\|u\|^{2^{*}-1}_{\nabla^{m}_{\alpha}}.
\end{eqnarray*}
This proves \eqref{w'k-estimate} for $k=1$. In addition, by assuming \eqref{w'k-estimate} for $1\leq k < m$,  it follows from \eqref{wk-estimate} that 
\begin{align*}
    |w_{k+1}'(r)| &= r^{-\alpha} \left| \int_{0}^{r} w_{k}(s) s^{\alpha} \,ds\right|\leq Cr^{-\alpha}  r^{\frac{\alpha-2m+1+4k}{2}}\|u\|^{2^{*}-1}_{\nabla^{m}_{\alpha}}\\
    &= C  r^{-\frac{\alpha+2m-1-4k}{2}}\|u\|^{2^{*}-1}_{\nabla^{m}_{\alpha}} = C  r^{-\frac{\alpha+2m+3-4(k+1)}{2}}\|u\|^{2^{*}-1}_{\nabla^{m}_{\alpha}}
\end{align*}
which represents \eqref{w'k-estimate} for $k+1$. Now, we proceed to prove \eqref{wkc2k}. From  \eqref{w'k-integral}, for $1\leq k \le m$ we have 
\begin{equation}\label{w''k-integral}
     w_{k}''(r) =  \alpha r^{-\alpha-1} \int_{0}^{r} w_{k-1}(s) s^{\alpha} \,ds - w_{k-1}(r)=-\frac{\alpha}{r}w^{\prime}_k(r)- w_{k-1}(r).
\end{equation}
From the above integral representation,  $\eqref{wkc2k}$ holds for $k=1$. In addition, if $\eqref{wkc2k}$ holds for all $1\le k<m$ then \eqref{w''k-integral} yields  $w_{k+1}''\in C^{2k}(0,\infty)$ which implies $w_{k+1}\in C^{2k+2}(0,\infty)$. Thus, by induction again we can conclude that \eqref{wkc2k} holds. Finally, from \eqref{w'k-integral} and \eqref{w''k-integral}, we have 
\begin{equation*}
    -\Delta_{\alpha}w_{\ell}= -\big(w_{\ell}''+\frac{\alpha}{r} w_{\ell}'\big)=w_{\ell-1}, \;\; \mbox{for all}\;\; 1\le \ell\le m.
\end{equation*}
By iterative argument, for fixed $k$ with $1\le k\le m$ we can see that \eqref{Dita} holds. 
\end{proof}
\begin{proof}[Proof of Theorem \ref{Existencia de minimizante suave}] Let $u \in \mathcal{D}^{m,2}_{\infty}(\alpha)$ be a weak solution for
\eqref{problema m}. Set $w_0=|u|^{2^*-2}u$ and let $w_k$, $1\le k\le m$ be given by \eqref{wk-iteration}. By combining \eqref{wkc2k} and \eqref{Dita}  the function $w:=w_{m}$ satisfies
\begin{equation}\label{wm-RS}
\begin{aligned}
 w\in C^{2m}(0,\infty)\;\;\mbox{and}\;\; (-\Delta_{\alpha})^{m}w=|u|^{2^*-2}u.
\end{aligned}
\end{equation}
In addition,   \eqref{wk-estimate}, \eqref{w'k-estimate} and \eqref{Dita} imply
\begin{equation}\label{Wm}
\begin{aligned}
& |\Delta_{\alpha}^{j} w| \leq C r^{- \frac{\alpha-2m+1+4j}{2}}\|u\|^{2^{*}-1}_{\nabla^{m}_{\alpha}}\\
&|(\Delta_{\alpha}^{j} w)'| \leq C r^{- \frac{\alpha-2m+3+4j}{2}}\|u\|^{2^{*}-1}_{\nabla^{m}_{\alpha}},
\end{aligned}
\end{equation}
for all $1\le j\le k\le m$. From \eqref{Wm}, we conclude that $w$ satisfies \eqref{0-Ulimites} and \eqref{infinty-Ulimites}. Further, \eqref{wk-iteration} and \eqref{Dita} yield
\begin{equation}\label{wm-IT}
(-\Delta_{\alpha})^{j} w(r)=\int_{r}^{\infty} t^{-\alpha}\left(\int_{0}^{t}s^{\alpha}(-\Delta_{\alpha})^{j+1}w(s)  \,ds\right)dt, \quad r>0.
\end{equation}
From \eqref{wm-RS}, \eqref{Wm} and \eqref{wm-IT}, it remains to show that $w:=w_m\in \mathcal{D}^{m,2}_{\infty}(\alpha)$ and $u=w$. First, we will show that $w:=w_m\in \mathcal{D}^{m,2}_{\infty}(\alpha)$. From \eqref{wm-RS}, we have  $w\in AC^{m-1}_{loc}(0,\infty)$.  By combining  \cite[Lemma 2.9]{JN-JF} with  \eqref{wk-estimate} and \eqref{Wm}  and by using induction argument we get
\begin{equation}\label{esti cond front}
    \lim_{r \to \infty} w^{(j)}(r) = 0, \quad j= 0, 1, \cdots, m-1.
\end{equation}
It follows that $w \in AC^{m-1}_{\mathrm{R}}(0, \infty)$. Now,  to conclude that $w\in\mathcal{D}^{m,2}_{\infty}(\alpha)$ we must prove 
\begin{equation}\label{w-mora}
    \int_{0}^{\infty} |\nabla^{m}_{\alpha} w|^{2} r^{\alpha}\, dr < \infty.
\end{equation}
Let $q_{0} = 2^{*}/(2^{*}-1)$ and, for $k = 1, \cdots, m$, set 
$$q_{k} := \dfrac{q_{k-1}(\alpha+1)}{\alpha-2q_{k-1}+1} = \dfrac{2(\alpha+1)}{\alpha+2m+1-4k}.$$
Then, for $w_k$ such as in \eqref{wk-iteration}, we claim that
   \begin{equation}\label{item-(a)}
    \int_{0}^{\infty} |w_{k}''|^{q_{k-1}} r^{\alpha}\, dr \leq C\int_{0}^{\infty} |w_{k-1}|^{q_{k-1}} r^{\alpha} \, dr
\end{equation}
and
    \begin{equation}\label{item-(b)}
    \int_{0}^{\infty} |w_{k}|^{q_{k}} r^{\alpha}\, dr \leq C\left(\int_{0}^{\infty} |w_{k-1}|^{q_{k-1}} r^{\alpha} \, dr\right)^{\frac{q_{k}}{q_{k-1}}}. 
\end{equation}
Indeed, for each $1\leq k \leq m$,  let $$v_{k}(r) := \int_{0}^{r} w_{k-1}(s) s^{\alpha} \, ds.$$
Since $v_{k}(0) = 0$, we are able to apply \cite[Proposition 2.3 (b)]{JN-JF} with the choices $\theta=\alpha-q_{k-1}(\alpha+1)$, $p=q_{k-1}= 2(\alpha+1)/(\alpha+2m+5-4k)>1$ and $\gamma = \theta + p$, to obtain
\begin{eqnarray*}
  \int_{0}^{\infty} \left| r^{-(\alpha+1)} \int_{0}^{r} w_{k-1}(s) s^{\alpha} \, ds \right|^{q_{k-1}} r^{\alpha}\, dr &=& \int_{0}^{\infty} \left| v_{k}(r) \right|^{q_{k-1}} r^{\alpha-q_{k-1}(\alpha+1)}\, dr \nonumber\\
  &\leq& C \int_{0}^{\infty} \left| v_{k}'(r) \right|^{q_{k-1}} r^{\alpha-q_{k-1}\alpha}\, dr \nonumber\\
  &=& C \int_{0}^{\infty} \left| w_{k-1}(r) \right|^{q_{k-1}} r^{\alpha}\, dr.
\end{eqnarray*}
Hence, 
\begin{eqnarray*}\label{m=2l perak}
\int_{0}^{\infty} |w_{k}''|^{q_{k-1}} r^{\alpha}\, dr &=& \int_{0}^{\infty} \left| - \alpha r^{-(\alpha+1)} \int_{0}^{r} w_{k-1}(s) s^{\alpha} \, ds + w_{k-1}\right|^{q_{k-1}} r^{\alpha}\, dr \nonumber\\
&\leq& C\int_{0}^{\infty} \left| r^{-(\alpha+1)} \int_{0}^{r} w_{k-1}(s) s^{\alpha} \, ds \right|^{q_{k-1}} r^{\alpha}\, dr + C \int_{0}^{\infty} |w_{k-1}|^{q_{k-1}} r^{\alpha} \, dr \nonumber\\
&\leq& C \int_{0}^{\infty} \left| w_{k-1}(r) \right|^{q_{k-1}} r^{\alpha}\, dr
\end{eqnarray*}
which proves \eqref{item-(a)}. To prove \eqref{item-(b)}, we use \cite[Proposition 2.3 (a)]{JN-JF} with the choices $\theta = \gamma = \alpha$ and $p = q_{k-1}>1$, to get 
\begin{equation}\label{M pera}      \int_{0}^{\infty}  |w_{k}(r)|^{q_{k}} r^{\alpha}\, dr  \leq C \left(\int_{0}^{\infty} \left| w_{k}''(r) \right|^{q_{k-1}} r^{\alpha}\, dr\right)^{\frac{q_{k}}{q_{k-1}}}, \quad q_{k} = q_{k-1}^{*} := \dfrac{q_{k-1}(\alpha+1)}{\alpha-2q_{k-1}+1}.
\end{equation}
From \eqref{item-(a)} and \eqref{M pera} we obtain \eqref{item-(b)}.

Now, the proof of \eqref{w-mora} is divided into two parts. If $m=2l$ is an even number, it follows from \eqref{Dita} that $\nabla_{\alpha}^{m}w = \Delta_{\alpha}^{l}w = (-1)^{l} w_{l}$. Further, by using \eqref{item-(b)}
\begin{equation}\label{salto 1}
 \begin{aligned}
    \int_{0}^{\infty} |w_{l}|^{q_{l}} r^{\alpha}\, dr &\leq C\left(\int_{0}^{\infty} |w_{l-1}|^{q_{l-1}} r^{\alpha} \, dr\right)^{\frac{q_{l}}{q_{l-1}}}\\
    &\leq C\left(\int_{0}^{\infty} |w_{l-2}|^{q_{l-2}} r^{\alpha} \, dr\right)^{\frac{q_{l}}{q_{l-2}}}\\
    &\leq  \cdots\\
&\leq C\left(\int_{0}^{\infty} |w_{0}|^{q_{0}} r^{\alpha} \, dr\right)^{\frac{q_{l}}{q_{0}}},
\end{aligned}   
\end{equation}
which yields
\begin{equation*}
    \int_{0}^{\infty} |\nabla_{\alpha}^{m}w|^{2} r^{\alpha}\, dr    = \int_{0}^{\infty} |\Delta_{\alpha}^{l} w|^{2} r^{\alpha}\, dr \leq C\left(\int_{0}^{\infty} |u|^{2^{*}} r^{\alpha} \, dr\right)^{\frac{2}{q_{0}}}< \infty.
\end{equation*}
If $m=2l+1$ is an odd number, \eqref{Dita} yields $\nabla_{\alpha}^{m} w = (\Delta_{\alpha}^{l}w)' = (-1)^{l}w_{m-l}'= (-1)^{l}w'_{l+1}$. As in \eqref{salto 1}, we have
    \begin{equation}\label{M-preto}
    \int_{0}^{\infty} |w_{l}|^{q_{l}} r^{\alpha}\, dr \leq C\left(\int_{0}^{\infty} |u|^{2^*} r^{\alpha} \, dr\right)^{\frac{q_{l}}{q_{0}}}, 
\end{equation}
where $q_{l} = q_{l-1}(\alpha+1)/(\alpha-2q_{l-1}+1)$ 
and $q_{0} = 2^{*}/(2^{*}-1)$. On the other hand, from \cite[Proposition 2.3 (a)]{JN-JF} with $\theta = \gamma = \alpha$ and $p = q_{l}>1$, we can write
\begin{equation}\label{M-verde}      \int_{0}^{\infty}  |w'_{l+1}(r)|^{q^{*}_{l}} r^{\alpha}\, dr  \leq C \left(\int_{0}^{\infty} \left| w_{l+1}''(r) \right|^{q_{l}} r^{\alpha}\, dr\right)^{\frac{q^{*}_{l}}{q_{l}}}, \quad  q_{l}^{*} = \dfrac{q_{l}(\alpha+1)}{\alpha-q_{l}+1}.
\end{equation}
From \eqref{item-(a)} with $k=l+1$ we can see that
\begin{equation}\label{passo-3}
    \int_{0}^{\infty}  |w''_{l+1}(r)|^{q_{l}} r^{\alpha}\, dr  \leq C \int_{0}^{\infty} \left| w_{l}(r) \right|^{q_{l}} r^{\alpha}\, dr.
\end{equation}
Noticing that $q_{l}^{*}=q_{l}(\alpha+1)/(\alpha-q_{l}+1)=2$, the estimates \eqref{M-preto},
 \eqref{M-verde} and \eqref{passo-3} imply 
\begin{equation}\nonumber
 \int_{0}^{\infty} |\nabla_{\alpha}^{m}w|^{2} r^{\alpha}\, dr=\int_{0}^{\infty}  |w'_{l+1}(r)|^{q^{*}_{l}} r^{\alpha}\, dr \leq C\left(\int_{0}^{\infty} |u|^{2^*} r^{\alpha} \, dr\right)^{\frac{q_{l}}{q_{0}}}.
\end{equation}
This completes the proof of $w\in \mathcal{D}^{m,2}_{\infty}(\alpha)$. Next, we will prove the identity $u=w$. First, we prove the following
\begin{equation}\label{m=2l est 2.3000}
     \int_{0}^{\infty} \nabla^{m}_{\alpha} w \nabla^{m}_{\alpha} v r^{\alpha} \,dr = \int_{0}^{\infty} |u|^{2^{*}-2}u v r^{\alpha}\, dr,  \;\;\mbox{for all}\;\; v \in \mathcal{D}^{m,2}_{\infty}(\alpha).
\end{equation}
By density (cf. Lemma~\ref{lemma crucial D}), it is sufficient to show \eqref{m=2l est 2.3000} for $v\in\Upsilon$. 
For each $v\in\Upsilon$, by using \eqref{w'k-estimate} and Fubini's theorem, for  $k=1,\cdots, m$  we obtain

\begin{equation}\label{esti laranja}
\begin{aligned}
     \int_{0}^{\infty} w_{k}' (\Delta_{\alpha}^{k-1}v)'r^{\alpha}\, dr&= -\int_{0}^{\infty} (\Delta_{\alpha}^{k-1}v)'(r) \left(\int_{0}^{r} w_{k-1}(s) s^{\alpha}\, ds\right) \, dr  \\
   &= -\int_{0}^{\infty}  \int_{0}^{r} (\Delta_{\alpha}^{k-1}v)'(r) w_{k-1}(s) s^{\alpha}\, ds \, dr \\
    &= \int_{0}^{\infty} w_{k-1}(s) s^{\alpha} \left(-\int_{s}^{\infty} (\Delta_{\alpha}^{k-1}v)'(r) \, dr\right) \, ds \\
    &=  \int_{0}^{\infty} w_{k-1}(s) \Delta_{\alpha}^{k-1}v(s) s^{\alpha}  \, ds,
\end{aligned}
\end{equation}
where we have used  that $(\Delta_{\alpha}^{k-1}v)(r)$ has compact support in $\mathbb{R}$. In addition, we have 
$$\lim_{r \to \infty } w_{k}(r) r^{\alpha} (\Delta_{\alpha}^{k-1}v)'(r)=0$$
and from \eqref{wk-estimate} we also can conclude that (recall \eqref{SSS-condi})
$$\lim_{r \to 0} w_{k}(r) r^{\alpha} (\Delta_{\alpha}^{k-1}v)'(r)=0.$$
Thus, we can apply integration by parts to get
\begin{equation}\label{wk'-IT}
\begin{aligned}
    \int_{0}^{\infty} w_{k}' (\Delta_{\alpha}^{k-1}v)'r^{\alpha}\, dr&= - \int_{0}^{\infty} w_{k} [(\Delta_{\alpha}^{k-1}v)'r^{\alpha}]'\, dr=  - \int_{0}^{\infty} w_{k} \Delta_{\alpha}^{k}v r^{\alpha}\, dr.
\end{aligned}    
\end{equation}
By combining \eqref{esti laranja} and \eqref{wk'-IT}, it follows that
\begin{equation}\label{esti pera}
    \int_{0}^{\infty} w_{k} \Delta_{\alpha}^{k}v r^{\alpha}\, dr = -\int_{0}^{\infty} w_{k-1}\Delta_{\alpha}^{k-1}v r^{\alpha}\, dr, \quad k=1, \cdots, m.
\end{equation}
By iterating \eqref{wk'-IT} and \eqref{esti pera}, for $k=1, \cdots, m$ we can write
\begin{align}
   & \int_{0}^{\infty} w_{k} \Delta_{\alpha}^{k}v r^{\alpha}\, dr=(-1)^{k}\int_{0}^{\infty} w_{0}v r^{\alpha}\, dr,\label{esti pera-completo}\\
   & \int_{0}^{\infty} w_{k}' (\Delta_{\alpha}^{k-1}v)'r^{\alpha}\, dr=(-1)^{k}\int_{0}^{\infty} w_{0}v r^{\alpha}\, dr \label{esti pera-completoII}.
\end{align}   
Firstly, suppose that $m=2l$ is an even integer number. From \eqref{Dita} with $j=l$ and $k=m$,  $(-\Delta_{\alpha})^{l} w = w_{m-l}= w_{l}$ and so $\Delta_{\alpha}^{l} w =(-1)^{l} w_{l}$. Hence, by  \eqref{esti pera-completo} we have
\begin{equation}\label{EVEN}
   \int_{0}^{\infty} \nabla^{m}_{\alpha} w \nabla_{\alpha}^{m}v r^{\alpha}\, dr  =\int_{0}^{\infty} \Delta_{\alpha}^{l}w \Delta_{\alpha}^{l}v r^{\alpha}\, dr = \int_{0}^{\infty} w_{0} v r^{\alpha}\, dr, \quad \forall v \in \Upsilon. 
\end{equation}
Analogously, if  $m=2l+1$ is an odd integer number. From \eqref{Dita} we also have  $(\Delta_{\alpha}^{l}w)' = (-1)^{l} w'_{l+1}$. By  \eqref{esti pera-completoII} we have
\begin{equation}\label{ODD}
  \int_{0}^{\infty} \nabla^{m}_{\alpha} w \nabla_{\alpha}^{m}v r^{\alpha}\, dr  = \int_{0}^{\infty} (\Delta_{\alpha}^{l}w)' (\Delta_{\alpha}^{l}v)' r^{\alpha}\, dr = \int_{0}^{\infty} w_{0} v r^{\alpha}\, dr, \quad \forall\, v \in \Upsilon.  
\end{equation}
Thus, \eqref{m=2l est 2.3000} follows from \eqref{EVEN} and \eqref{ODD}. By comparing \eqref{solution weak} and \eqref{m=2l est 2.3000}, we get
\begin{equation}\nonumber
    \langle u ,v\rangle = \langle w, v\rangle, \quad \forall v \in \mathcal{D}^{m,2}_{\infty}(\alpha)
\end{equation}
which provides $u=w$.
\end{proof}
As a consequence of the argument used in the proof of Theorem~\ref{Existencia de minimizante suave}, we have the following.
\begin{corollary} \label{UCoroLemmaR} If $u \in \mathcal{D}^{m,2}_{\infty}(\alpha)$ is a weak solution
\eqref{problema m}, then 
\begin{align}
     &u_j(r) \leq C r^{- \frac{\alpha+2m+1-4(m-j)}{2}}\|u\|^{2^{*}-1}_{\nabla^{m}_{\alpha}} \label{Uwm-estimate}\\
     &u^{\prime}_j(r) \leq C r^{- \frac{\alpha+2m+3-4(m-j)}{2}}\|u\|^{2^{*}-1}_{\nabla^{m}_{\alpha}}\label{Uw'm-estimate},
\end{align}
where $u_j=(-\Delta_{\alpha})^{j}u$.
\end{corollary}
\begin{proof}
    In fact, we have shown that $u=w_m$, where $w_m$ is given by \eqref{wk-iteration}. Thus, from \eqref{Dita} $u_j= (-\Delta_{\alpha})^{j}w_m=w_{m-j}$, for $1\le j\le m $. Consequently, by using \eqref{wk-estimate} and \eqref{w'k-estimate}  we can write \eqref{Uwm-estimate} and \eqref{Uw'm-estimate}.
\end{proof}
\begin{remark}\label{u-monotona}
If we assume in Theorem~\ref{Existencia de minimizante suave} that the weak solution \( u \in \mathcal{D}^{m,2}_{\infty}(\alpha) \) is positive, then \( u \) is a decreasing function. Indeed, in this case, \( w_0 = u^{2^{*}-1} > 0 \), and all functions \( w_k \), \( 1 \leq k \leq m \), defined by \eqref{wk-iteration}, are positive and decreasing on \( (0,\infty) \). In particular, \( u = w_m \) is a decreasing function. Furthermore, the boundary conditions \eqref{0-Ulimites} can be improved, as shown in Corollary~\ref{corollaryreferee} below.
\end{remark}

\begin{lemma} \label{lemma-iterationRev} Let $u \in \mathcal{D}^{m,2}_{\infty}(\alpha)$ with $u>0$ in $[0, \infty)$ be a weak solution to $(-\Delta_{\alpha})^{m} u = u^{2^{*}-1}$ and let $w_k$ be given by \eqref{wk-iteration}. Then, for $k=1,\cdots, m$ we have 
\begin{equation}\label{wkCadeia}
    \left\{\begin{aligned}
      & w_k(0)<\infty\\
      &w^{\prime}_k(0)=0\\
      &w^{(2)}_k(0)=-\frac{1}{\alpha+1}w_{k-1}(0).
    \end{aligned}\right.
\end{equation}
Additionally, we have $u^{(2)}(0)<\infty$ and
$u^{(1)}(0) = u^{(3)}(0) = 0.$
\end{lemma}
\begin{proof}
Firstly, we have $w_0(0)=u^{2^{*}-1}(0)<\infty$. From \eqref{wk-iteration}, the  L'Hospital rule yields
\begin{align}\label{w1Prime/r}
 \lim_{r\to 0}\frac{w_1^{\prime}(r)}{r}=-\lim_{r\to 0}\frac{\int_{0}^{r}s^{\alpha}w_0(s)\mathrm{d}s}{r^{\alpha+1}}=-\frac{1}{\alpha+1}\lim_{r\to 0}w_0(r)=-\frac{w_0(0)}{\alpha+1}.
\end{align}
Directly from \eqref{w1Prime/r}, we get 
\begin{equation}\label{w1PrimeNo0}
   w^{\prime}_1(0)=0. 
\end{equation} Thus, \eqref{w1Prime/r} implies
\begin{align}\label{w12Prime}
 w_1^{(2)}(0)=\lim_{r\to 0}\frac{w_1^{\prime}(r)}{r}=-\frac{w_0(0)}{\alpha+1}.
\end{align}
Now, we claim that 
\begin{equation}\label{w1Late}
   w_1(0)= \lim_{r\to 0}w_1(r)<\infty.
\end{equation}
In fact, as noted in Remark~\ref{u-monotona}, we have that $w_1$ is positive and decreasing on $(0,\infty)$. Thus,  $\lim_{r\to 0}w_1(r)=w_1(0)\in (0, \infty]$. We will show that $w_1(0)<\infty$.  By contradiction,  suppose that there exists a  sequence  $(r_n)\subset (0,\infty)$ with $\lim_{n\to\infty}r_n=0$ and  
\begin{equation}\label{contra-limite}
    \lim_{n\to\infty} w_1(r_n)= \infty.
\end{equation}
From \eqref{w1PrimeNo0}, there is $\delta>0$ such that $|w_1^{\prime}(s)|\le 1$ for $s\in (0,\delta]$. Then
\begin{equation}\label{TFC}
    |w_1(r_n)-w_1(r_m)|=\Big|\int_{r_n}^{r_m}w^{\prime}_{1}(s)\mathrm{d}s\Big|\le |r_m-r_n|
\end{equation}
for $m, n$ large enough for $r_n,r_m\in (0, \delta]$. This contradicts \eqref{contra-limite} and conclude \eqref{w1Late}. From \eqref{w1PrimeNo0}, \eqref{w12Prime} and \eqref{w1Late}, we have that \eqref{wkCadeia} holds for $k=1$. By induction, suppose that \eqref{wkCadeia} holds for some $k\le m-1$. Then \eqref{wk-iteration} yields 
\begin{align}\label{wk+1Prime/r}
 \lim_{r\to 0}\frac{w_{k+1}^{\prime}(r)}{r}=-\lim_{r\to 0}\frac{\int_{0}^{r}s^{\alpha}w_k(s)\mathrm{d}s}{r^{\alpha+1}}=-\frac{1}{\alpha+1}\lim_{r\to 0}w_k(r)=-\frac{w_k(0)}{\alpha+1}.
\end{align}
From \eqref{wk+1Prime/r}, we get 
\begin{equation}\label{wk+1PrimeNo0}
   w^{\prime}_{k+1}(0)=0 
\end{equation} 
and
\begin{align}\label{wk+12Prime}
 w_{k+1}^{(2)}(0)=\lim_{r\to 0}\frac{w_{k+1}^{\prime}(r)}{r}=-\frac{w_{k}(0)}{\alpha+1}.
\end{align}
In addition, since $w_{k+1}$ is positive and decreasing on $(0,\infty)$, arguing as in  \eqref{w1Late}, \eqref{contra-limite} and \eqref{TFC}, we have that \eqref{wk+1PrimeNo0} implies
\begin{equation}\label{wk+1Late}
   w_{k+1}(0)= \lim_{r\to 0}w_{k+1}(r)<\infty
\end{equation}
which completes the induction argument. Finally, from $u = w_{m}$ and by using \eqref{wkCadeia}, we get $$u^{(1)}(0) = w^{(1)}_m(0)=0\;\;\mbox{and}\;\; u^{(2)}(0) = -\dfrac{1}{\alpha+1} w_{m-1}(0).$$ Now, from \eqref{wk-iteration}
\begin{equation}\label{new-D3}
\begin{aligned}
     w_{k}^{(3)}(r) &=  -\alpha (\alpha+1) r^{-\alpha-2} \int_{0}^{r} w_{k-1}(s) s^{\alpha} \,ds + \alpha r^{-1}  w_{k-1}(r) - w'_{k-1}(r)\\
     &= -\alpha (\alpha+1) r^{-\alpha-2} \int_{0}^{r} \left(w_{k-1}(s) -  w_{k-1}(r)\right) s^{\alpha} \,ds - w'_{k-1}(r).
\end{aligned}
\end{equation}
Using \eqref{new-D3} we obtain
\begin{equation*}
    \begin{aligned}
        u^{(3)}(r)=w_{m}^{(3)}(r) &= -\alpha (\alpha+1) r^{-\alpha-2} \int_{0}^{r} \left(w_{m-1}(s) -  w_{m-1}(r)\right) s^{\alpha} \,ds - w'_{m-1}(r).
    \end{aligned}
\end{equation*}
Given $\epsilon>0$, by $w'_{m-1}(0) = 0$ implies that there exists $\delta > 0$ such that 
$$|w_{m-1}'(r)| \leq \epsilon \quad \text{and} \quad |w_{m-1}(r) - w_{m-1}(s)| \leq \epsilon (r-s) \quad \text{for}\quad 0 < s \leq r < \delta.$$ Then, for $r \in (0, \delta)$ we have
\begin{equation*}
    \begin{aligned}
        |u^{(3)}(r)| &\leq  \epsilon \alpha (\alpha+1) r^{-\alpha-2} \int_{0}^{r} r s^{\alpha} \,ds + \epsilon = (\alpha+1) \epsilon.
    \end{aligned}
\end{equation*}
This implies that $u^{(3)}(0) = 0$.
\end{proof}
As a consequence of Lemma~\ref{lemma-iterationRev}, we have the following: 
\begin{corollary}\label{corollaryreferee} Let $u \in \mathcal{D}^{m,2}_{\infty}(\alpha)$ with $u>0$ in $[0, \infty)$ be a weak solution to $(-\Delta_{\alpha})^{m} u = u^{2^{*}-1}$ and let $u_j=(-\Delta_\alpha)^{j}u$, for $j=0,1,\cdots, m$. Then 
\begin{equation}\label{controle-singularI}
\left\{\begin{aligned}
 &u_j(0)<\infty, &\mbox{for} &\;\; j=0, 1,\cdots, m\\
 &u^{\prime}_{j}(0)=0,&\mbox{for}&\;\; j=1,\cdots, m-1\\
 &u^{(2)}_{j}(0)=-\frac{1}{\alpha+1}u_{j+1}(0),&\mbox{for}&\;\; j=1,\cdots, m-1.\\
\end{aligned}\right.
\end{equation}
\end{corollary}
\begin{proof}
By assumption, we have  $u_0(0)=u(0)<\infty$ and $u_m(0)=u^{2^*-1}(0)$. 
As noted in   Remark~\ref{u-monotona}, we have  that \( u = w_m \), where $w_m$ is given in \eqref{wk-iteration}. From \eqref{Dita}, it follows that 
\begin{equation}\label{deuparaWW}
  u_j=w_{m-j},\;\;\;   u^{\prime}_j=w^{\prime}_{m-j},\;\; \mbox{and}\;\; u^{(2)}_j=w^{(2)}_{m-j},\;\; \mbox{for}\;\; j=1, \cdots, m.
\end{equation}
Now, \eqref{controle-singularI} follows from \eqref{wkCadeia} and \eqref{deuparaWW}. 
\end{proof}
To determine the sign of a minimizers, we establish an estimate for solutions that change sign. The following lemma was inspired by \cite[Lemma 7.22]{Gazzola-Grunau-Sweers}. 
\begin{lemma}\label{sinal de minimizante}
    Assume that there exists a nodal solution 
 $u \in \mathcal{D}^{m,2}_{\infty}(\alpha)$ of the equation \eqref{problema m}. Then,
    \begin{equation}\label{lemma 2.9 - estimativa}
        \|u\|^{2}_{\nabla^{m}_{\alpha}} \geq 2^{\frac{2m}{\alpha+1}} \mathcal{S} \|u\|^{2}_{L^{2^{*}}_{\alpha}},
    \end{equation}
    where $\mathcal{S}$ is given by \eqref{c0-e40}.
\end{lemma}
\begin{proof}
Let us consider the closed convex nonempty cone 
$$\mathcal{H} = \{u \in \mathcal{D}^{m,2}_{\infty}(\alpha) \,:\; u \geq  0 \quad \text{a.e. in}\,\, (0, \infty) \}$$
and its dual cone by
$$\mathcal{H^{*}} = \{u \in \mathcal{D}^{m,2}_{\infty}(\alpha) \,:\; \langle u, v\rangle \leq 0, \;\; \forall\; v \in \mathcal{H}\}.$$
We will claim that 
\begin{equation}\label{condicao cone}
    \mathcal{H^{*}} \subseteq -\mathcal{H}.
\end{equation}
Indeed, let $\varphi \in \mathcal{H}\cap \Upsilon$ with $\Upsilon = \{ u_{|_{[0, \infty)}} \,: \; u \in C^{\infty}_{c}(\mathbb{R})\}$. Consider the  problem
\begin{equation}\label{PV-lema2.8}
(- \Delta_{\alpha})^{m} y = \varphi\;\; \mbox{in}\;\; (0, \infty).
\end{equation}  
By the Riesz Representation Theorem, there exists a unique weak solution $y_{\varphi}\in\mathcal{D}^{m,2}_{\infty}(\alpha)$ to \eqref{PV-lema2.8}. The same argument in the proof of Theorem \ref{Existencia de minimizante suave} allows us to conclude that $y_{\varphi}\in C^{2m}(0,\infty)$ and satisfies \eqref{PV-lema2.8} in the usual sense and the boundary conditions
\begin{equation}\label{y0-Ulimites}
\lim_{r\to 0}r^{\alpha}\Delta_{\alpha}^{j} y_{\varphi} (r) = \lim_{r\to 0}r^{\alpha}((\Delta_{\alpha})^{j}y_{\varphi})^{\prime} (r) = 0,\;\; \mbox{for}\;\; j= 0, 1, \cdots, m-1
\end{equation}
and
\begin{equation}\label{yinfinite-Ulimites}
\lim_{r\to \infty}\Delta_{\alpha}^{j} y_{\varphi} (r) = \lim_{r\to \infty}((\Delta_{\alpha})^{j}y_{\varphi})^{\prime} (r) = 0,\;\; \mbox{for}\;\; j= 0, 1, \cdots, m.
\end{equation}
In order to prove that $y_{\varphi}$ preserves sign of $\varphi$, we rewrite \eqref{PV-lema2.8} as the system
\begin{equation}\label{PV-lema2.8-2}
 \left\{\begin{aligned}
& - \Delta_{\alpha} y_{k} = y_{k+1},\;\; &\mbox{for}&\;\; k=0, 1, \dots, m-2,\\
& - \Delta_{\alpha} y_{m-1} = \varphi,\;\; &\mbox{in}&\;\; (0, \infty),
\end{aligned}\right.
\end{equation}
where $y_{k} = (- \Delta_{\alpha})^{k} y$ with $y\in C^{2m}(0,\infty)\cap \mathcal{D}^{m,2}_{\infty}(\alpha)$ satisfying \eqref{y0-Ulimites} and \eqref{yinfinite-Ulimites}. From \eqref{PV-lema2.8-2} and \eqref{y0-Ulimites} we can deduce that $$y'_{m-1}(r) = - \frac{1}{r^{\alpha}}\int_{0}^{r} \varphi(s) s^{\alpha}\, ds, \;\; \text{for} \;\; r>0.$$
Since  $\varphi \geq 0$ in $(0, \infty)$ we get that $y_{m-1}$ is a nonincreasing function on $(0,\infty)$. From \eqref{yinfinite-Ulimites}, we have 
$$\lim_{r\to \infty}y_{m-1} (r)=0$$  
which implies $y_{m-1}\geq 0$ in $(0,\infty)$. Now, for each $k=0,1,2,\dots, m-2$,  from \eqref{PV-lema2.8-2} and \eqref{y0-Ulimites}, we obtain 
\begin{equation}\label{ykI}
   y'_{k}(r) = - \frac{1}{r^{\alpha}}\int_{0}^{r} y_{k+1}(s) s^{\alpha}\, ds \quad \text{in} \,\, (0, \infty). 
\end{equation}
Hence, $y_{m-1}\ge 0$ and \eqref{ykI} yield $y_{m-2}$ nonincreasing on $(0, \infty)$ and \eqref{yinfinite-Ulimites} implies $y_{m-2}\geq 0$ in $(0,\infty)$. By repeating the  argument,
we get that the functions $y_{m-3}, \dots,y_2, y_{1}, y_0$ are nonnegative on $(0,\infty)$. Hence, since $y_0=y_{\varphi}$ and  $y_{k+1}=-\Delta_{\alpha}y_k$, $k=0,1, \dots, m-2$  solves the system \eqref{PV-lema2.8-2}, we ensure $y_{\varphi}\ge0$ or $y_{\varphi} \in \mathcal{H}$. Now, we are in a position to prove \eqref{condicao cone}. Indeed,  for any $w\in \mathcal{H}^{*}$,  we get $\langle w, y_{\varphi}\rangle\le 0$. By using that $y_{\varphi}$ is weak solution of  \eqref{PV-lema2.8}, we obtain 
\begin{eqnarray*}
    \int_{0}^{\infty} w \varphi r^{\alpha}\, dr  = \langle w, y_{\varphi}\rangle\le 0,\;\; \mbox{for all}\;\; \varphi \in  \mathcal{H}\cap \Upsilon. 
\end{eqnarray*}
It follows that $w\le 0$ in $(0,\infty)$ or $w \in -\mathcal{H}$ as claimed in \eqref{condicao cone}. Let $u$ be a nodal weak solution of \eqref{problema m}, by the decomposition of a Hilbert space into dual cones due to Moreau \cite{Moreau} (see also \cite[Theorem 3.4]{Gazzola-Grunau-Sweers}) there exists a unique $(u_{1}, u_2)\in \mathcal{H}\times\mathcal{H^{*}}$ satisfying
\begin{equation}\label{decomposição cones}
    u = u_{1}+u_{2} \quad\text{and}\quad \langle u_{1}, u_{2} \rangle = 0.
\end{equation}
We claim that for $i=1,2$ we have
\begin{equation}\label{lema2.9 - est1}
u_{i} \not\equiv 0 \quad \text{and}\quad    |u|^{2^{*}-2} u u_{i} \leq |u_{i}|^{2^{*}} \;\; \text{a.e. in} \,\, (0, \infty).
\end{equation}
Since $u$ changes sign, \eqref{condicao cone} ensures $u_{i} \not\equiv 0$, for $i=1,2$. To prove the second part of \eqref{lema2.9 - est1}, we first consider $i=1$  and analyze two cases. If $r\in (0,\infty)$ is such that $u(r) \leq 0$, then $\eqref{lema2.9 - est1}$ is trivial. On the other hand, since \eqref{condicao cone} yields $u_2\le 0$ , if $u(r) \geq 0$, we get $u(r) = u_{1}(r)+ u_{2}(r) \leq u_{1}(r)$ and \eqref{lema2.9 - est1} holds again. The case  $i=2$ is similar. By using \eqref{c0-e40}, \eqref{decomposição cones} and \eqref{lema2.9 - est1}, we have for $i=1,2$
\begin{eqnarray*}
    \mathcal{S}\|u_{i}\|^{2}_{L^{2^{*}}_{\alpha}} &\leq& \|u_{i}\|^{2}_{\nabla^{m}_{\alpha}} = \langle u, u_{i}\rangle_{\nabla^{m}_{\alpha}} 
    = \int_{0}^{\infty} |u|^{2^{*}-2} u u_{i} r^{\alpha}\, dr \\
    &\leq& \int_{0}^{\infty} |u_{i}|^{2^{*}} r^{\alpha}\, dr = \|u_{i}\|^{2^{*}}_{L_{\alpha}^{2^{*}}},
\end{eqnarray*}
which implies
\begin{equation}\label{lema2.9-3}
    \mathcal{S}^{\frac{\alpha-2m+1}{2m}} \leq \|u_{i}\|_{L_{\alpha}^{2^{*}}}^{2}, \;\; \text{for}\,\, i = 1, 2.
\end{equation}
From \eqref{solution weak} with $v=u$,  we deduce 
\begin{equation}\label{lema2.9 -4}
    \|u\|_{\nabla^{m}_{\alpha}}^{2} = \|u\|^{2^{*}}_{L^{2^{*}}_{\alpha}}.
\end{equation}
By combining \eqref{c0-e40}, \eqref{lema2.9 -4}, \eqref{decomposição cones} and \eqref{lema2.9-3}, we get
\begin{eqnarray*}
    \dfrac{\|u\|_{\nabla^{m}_{\alpha}}^{2}}{\|u\|^{2}_{L^{2^{*}}_{\alpha}}} &=& (\|u\|^{2}_{\nabla^{m}_{\alpha}})^{\frac{2m}{\alpha+1}} = \left(\|u_{1}\|^{2}_{\nabla^{m}_{\alpha}} + \|u_{2}\|^{2}_{\nabla^{m}_{\alpha}} \right)^{\frac{2m}{\alpha+1}}\\
    &\geq&  \left(\mathcal{S}\|u_{1}\|^{2}_{L^{2^{*}}_{\alpha}} + \mathcal{S}\|u_{2}\|^{2}_{L^{2}_{\alpha}} \right)^{\frac{2m}{\alpha+1}} \geq \mathcal{S} 2^{\frac{2m}{\alpha+1}},
\end{eqnarray*}
which yields \eqref{lemma 2.9 - estimativa}.
\end{proof}

\section{Proof of Theorem \ref{Teorema unicidade}}
\label{sec3}
The aim of this section is to prove Theorem~\ref{Teorema unicidade}. Our proof is divided into two main steps: 

\paragraph{Step~1:} For each $\epsilon>0$, the function $w_{\epsilon}$ in \eqref{c0-e13 +} solves equation \eqref{problema m}.

\paragraph{Step~2:} If $v\in \mathcal{D}^{m,2}_{\infty}(\alpha)\cap C^{2m}(0,\infty)$ is a positive $(2m-1)$-th nonsingular at $r=0$  solution to \eqref{problema m} and such that $v^{(i)}(0) = 0$, for  $i = 1, 3, \ldots, 2m-1$, then  $v=w_{\epsilon}$ for some $\epsilon>0$.

\subsection{Proof of Step $1$}\label{subsec122} Let us introduce the following suitable notation: For $x,y$ and $z$ real numbers, we set
\begin{eqnarray*}
 A(x, y, z) &=& x(x + y -1 ) +  z (z - 2x +1 - y),\\
 B(x, y, z) &=& 2x(x + y - 1) - z (2x + y +1),\\
 C(x, y) &=& x (x + y - 1).
\end{eqnarray*}
Now, let us consider the auxiliary function $v:(0, \infty) \to \mathbb{R}$ given by
\begin{equation}\label{c0-e9}
    v(r) = r^{\rho} (1+r^2)^{- \frac{\sigma}{2}}, 
\end{equation}
where $\rho\ge 0$ and $\sigma>0$ will be chosen later.  With this notation we have the following.
\begin{lemma}\label{c0-lema1} For each $\alpha>-1$ we have
$$\Delta_{\alpha} v  = (1+r^{2})^{- \frac{\sigma+4}{2}} \left[r^{\rho+2}A(\rho, \alpha, \sigma) + r^{\rho} B(\rho, \alpha, \sigma) + r^{\rho-2} C(\rho, \alpha)\right].$$
\end{lemma}
\begin{proof}
By a simple calculation, we get
\begin{equation*}
    v'(r) = (1+r^2)^{- \frac{\sigma+4}{2}} \left[r^{\rho-1} \rho + r^{\rho+1}(2\rho - \sigma) + r^{\rho+3}(\rho- \sigma)\right]
\end{equation*}
and
\begin{eqnarray*}
    v''(r) &=&  (1+r^2)^{- \frac{\sigma+4}{2}} \big\{r^{\rho-2} \rho(\rho-1) + r^{\rho}[2\rho(\rho-1) - (2\rho +1)\sigma] \\
    &+& r^{\rho+2}[\rho(\rho - 1)- \sigma (\sigma - 2\rho +1)] \big\}.
\end{eqnarray*}
From \eqref{laplace-geral}, we see
\begin{eqnarray*}
\Delta_{\alpha} v &=& (1+r^2)^{- \frac{\sigma+4}{2}} \big\{r^{\rho-2} [\rho(\rho-1) + \alpha \rho] + r^{\rho}[2\rho(\rho + \alpha -1) - \sigma (2\rho + \alpha +1)  ] \\
& & + r^{\rho+2}[\rho(\rho + \alpha - 1) + \sigma (\sigma - 2\rho + 1 - \alpha)] \big\} \\
&=&  (1+r^2)^{- \frac{\sigma+4}{2}} \big[r^{\rho+2} A(\rho, \alpha,\sigma)+ r^{\rho} B(\rho, \alpha,\sigma)+ r^{\rho-2} C(\rho, \alpha) \big].
\end{eqnarray*}
\end{proof}
\noindent To state our next result, we need to introduce some notations.  For $i\in\mathbb{Z}$, $m \in \mathbb{N}$ and $\alpha>-1$ satisfying $\alpha-2m+1>0$ and $j=1, \cdots, m$, we set
\begin{align*}
 D(i,j)&=\left\{\begin{aligned}
& 0 \;\;& \mbox{if}& \;\;i<0\;\;\mbox{or}\;\; i\ge j+1\\
& 1 \;\; &\mbox{if}&\;\; i=0\\
&\prod_{h= j-i +1}^{j} (m-h) \;\; &\mbox{if} &\;\; 1\leq i \leq j,
\end{aligned}\right.\\
 E(i,j)& =\left\{\begin{array}{llc}
\displaystyle\prod_{h= i}^{j-1} (\alpha +1+ 2h) &\mbox{if}\quad 0\leq i \leq j-1\\
1&\mbox{if}\quad i=j\\
0 &\mbox{if}\quad i \geq j+1\;\;\mbox{or}\;\; i<0
\end{array}\right. \\
  K_j &=  \displaystyle\prod_{h= 0}^{j-1} (\alpha -2m +1 + 2h),\\
 G(i, j) & = 2^{i} {j\choose i} K_{j} D(i,j) E(i,j),
\end{align*}
 with
\begin{equation*}
  \displaystyle{j\choose i}=\left\{
  \begin{aligned}
       &\frac{j!}{i!(j-i)!} \;\;&\mbox{if}& \;\; 0 \leq i \leq j\\
      &0 \;\; &\mbox{if}&\;\; i>j\;\;\mbox{or}\;\; i<0
  \end{aligned}\right.
 \end{equation*}
 where $\ell!$ denotes a factorial.
\begin{proposition}\label{c0-prop2}
Let $u : (0, \infty) \to \mathbb{R}$ defined by
\begin{equation}\label{c0-e10}
    u(r) = (1 + r^2)^{- \frac{\alpha - 2m +1}{2}}, \quad \text{with}\quad \alpha - 2m + 1 > 0.
\end{equation}
Then, for $j=1, \cdots , m$ holds
\begin{equation}\label{c0-e11}
    (- \Delta_{\alpha})^{j} u = (1+ r^2)^{- \frac{\alpha - 2m + 1 + 4j}{2}} \sum_{i=0}^{j} G(i, j) r^{2i}.
\end{equation}
\end{proposition}
\begin{proof}
For $j=1$, by choosing  $\sigma = \alpha -2m + 1$ and $\rho = 0$ in the Lemma \ref{c0-lema1}, we see that
\begin{align*}
    -\Delta_{\alpha} u &=(1 + r^2)^{- \frac{\alpha -2m + 1 +4}{2}} \left[ 2(\alpha - 2m + 1)(m-1)r^2 + (\alpha - 2m + 1)(\alpha + 1)\right]\\
     &= (1+ r^2)^{- \frac{\alpha -2m + 1 + 4}{2}} \displaystyle\sum_{i=0}^{1} G(i, j) r^{2i}.
\end{align*}
Thus, \eqref{c0-e11} holds for $j=1$. By induction, suppose that \eqref{c0-e11}  holds for  $1\le j<m-1$. Hence,   we obtain
\begin{equation*}\label{c0-e14}
\begin{aligned}
    (-\Delta_{\alpha})^{j+1} u =-\sum_{i=0}^{j} G(i,j)\Delta_{\alpha} v_{i}, \;\; \text{where} \;\; v_{i}(r) = r^{2i} (1+r^2)^{- \frac{\alpha - 2m + 1 + 4j}{2}}.
    \end{aligned}
\end{equation*}
By using Lemma~\ref{c0-lema1} with $\rho = 2i$ and $\sigma = \alpha -2m + 1 + 4j$, we can write 
\begin{align*}
 \Delta_{\alpha}v_i= (1+r^{2})^{- \frac{\sigma+4}{2}} \left[r^{2i+2}A(2i, \alpha, \sigma) + r^{2i} B(2i, \alpha, \sigma) + r^{2i-2} C(2i, \alpha)\right].   
\end{align*}
Note that 
\begin{align*}
    &\sum_{i=0}^{j} G(i,j)r^{2i+2}A(2i, \alpha, \sigma)=\sum_{i=1}^{j+1} G(i-1,j)A(2i-2, \alpha, \sigma)r^{2i}\\
    &\sum_{i=0}^{j} G(i,j) r^{2i-2} C(2i, \alpha)=\sum_{i=-1}^{j-1} G(i+1,j) C(2i+2, \alpha)r^{2i}.
\end{align*}
Hence, 
\begin{align*}
    \sum_{i=0}^{j} G(i,j)\Delta_{\alpha} v_{i}& =(1+r^{2})^{- \frac{\sigma+4}{2}}\sum_{i=1}^{j+1} G(i-1,j)A(2i-2, \alpha, \sigma)r^{2i} \\
    & + (1+r^{2})^{- \frac{\sigma+4}{2}}\sum_{i=0}^{j} G(i,j)  B(2i, \alpha, \sigma)r^{2i} \\
    &+(1+r^{2})^{- \frac{\sigma+4}{2}}\sum_{i=-1}^{j-1} G(i+1,j) C(2i+2, \alpha)r^{2i}.\\
\end{align*}
We observe that  $G(-1,j)=G(j+1,j)=G(j+2,j)=C(0,\alpha)=0$. Thus, we can extend the three right hand side summation above from $0$ to $j+1$ and write
\begin{align*}
    \sum_{i=0}^{j} G(i,j)\Delta_{\alpha} v_{i}& =(1+r^{2})^{- \frac{\sigma+4}{2}}\sum_{i=0}^{j+1} S(i,j)r^{2i},
\end{align*}
where
\begin{align*}
   S(i,j)= G(i-1,j)A(2i-2, \alpha, \sigma)+G(i,j)  B(2i, \alpha, \sigma)+G(i+1,j) C(2i+2, \alpha).
\end{align*}
Thus,
\begin{equation}\label{c0-e15}
\begin{aligned}
 (-\Delta_{\alpha})^{j+1} u &= -(1+r^{2})^{- \frac{\sigma+4}{2}}\sum_{i=0}^{j+1} S(i,j)r^{2i}.
\end{aligned}    
\end{equation}
It is easy to see that 
\begin{align*}
S(i,j)= K_j H(i,j),  
\end{align*}
where
\begin{equation}\label{c0-e17}
\begin{aligned}
H(i,j) &=  2^{i-1} {j \choose i-1} D(i-1,j) E(i-1, j) A(2i -2,\alpha, \sigma)\\
&+  2^{i} {j \choose i} D(i,j) E(i, j) B(2i, \alpha, \sigma)\\
&+  2^{i+1} {j \choose i+1} D(i+1,j) E(i+1, j) C(2i + 2,\alpha).
\end{aligned}
\end{equation}
Thus, we can rewrite \eqref{c0-e15} as follows
\begin{equation}\label{c0-e16}
   (-\Delta_{\alpha})^{j+1} u = -(1+r^2)^{- \frac{\sigma + 4}{2}} K_{j}\sum_{i=0}^{j+1} H(i,j) r^{2i}.
\end{equation}
We claim that
\begin{equation}\label{c0-e18}
    G(i, j+1) = - K_{j} H(i,j), \quad \forall i = 0, 1, \cdots, j+1.
\end{equation}
We will divide the proof of \eqref{c0-e18} into four cases:
\paragraph{\textbf{Case~1:}} $i = 1, 2, \cdots, j-1$.\\
\noindent By a simple calculation, we get
\begin{align*}
     &E(i-1, j) = (\alpha + 2i - 1)(\alpha + 2i + 1) E(i+1, j),\\
     &E(i,j)  = (\alpha + 2i + 1) E(i+1, j),\\
     & D(i+1, j) = (m-j+i)(m-j+i-1) D(i-1, j),\\
    &D(i, j) = (m-j+i-1) D(i-1, j).
\end{align*}
Thus, \eqref{c0-e17} yields
\begin{equation}\label{c0-e19}
\begin{aligned}
H(i,j) &=  D(i-1, j) E(i+1, j) \Bigg[ 2^{i-1} {j \choose i-1} (\alpha + 2i - 1)(\alpha + 2i + 1) A(2i -2,\alpha, \sigma) \\
&+  2^{i} {j \choose i} (\alpha + 2i + 1) (m-j+i-1) B(2i, \alpha,\sigma)\\
&+  2^{i+1} {j \choose i+1}  (m-j+i)(m-j+i-1) C(2i + 2, \alpha) \Bigg] .
\end{aligned}
\end{equation}
We also have the identities
\begin{eqnarray*}
     \displaystyle{j \choose i } &=& {j \choose i - 1}  \dfrac{j-i+1}{i},\\ 
     \displaystyle{j \choose i + 1} &=& {j \choose i - 1}  \dfrac{(j-i+1)(j-i)}{i(i+1)},\\
    A(2i-2, \alpha,\sigma) &=& 2(i-1)(2i + \alpha - 3 ) +  2(\alpha - 2m + 1 + 4j) (2j - 2i - m + 3),\\ 
    B(2i, \alpha,\sigma) &=& 4i (2i + \alpha - 1) - (\alpha - 2m + 1 + 4j) (4i + \alpha +1),\\
    C(2i+2, \alpha) &=& 2(i+1)(\alpha+2i + 1).
\end{eqnarray*}
Hence,  \eqref{c0-e19} implies
\begin{eqnarray*}\label{c0-e20}
H(i,j) &=&  \dfrac{2^{i}}{i} {j \choose i-1} D(i-1, j) E(i+1, j)  (\alpha + 2i +1)\times\nonumber\\
& \Bigg\{& i(\alpha + 2i - 1) \big[ (i-1)(2i + \alpha - 3 ) +  (\alpha - 2m + 1 + 4j) (2j - 2i - m + 3) \big] \nonumber\\
&+&  (j-i+1)(m-j+i-1) \big[ 4i (2i + \alpha - 1) - (\alpha - 2m + 1 + 4j) (4i + \alpha +1) \big]\nonumber\\
&+& 4 (m-j+i)(m-j+i-1) (j-i)(j-i+1) \Bigg\} .
\end{eqnarray*}
By arranging the terms in brackets, we can write
\begin{equation}\label{c0-e21}
  \begin{aligned}
H(i,j) &=  \dfrac{2^{i}}{i} {j \choose i-1} D(i-1, j) E(i+1, j)  (\alpha + 2i +1)\left[(\alpha + 1)^{2}L_{j} + (\alpha+1)M_{j} + Q_{j}\right],
\end{aligned}  
\end{equation}
where
\begin{equation}\label{c0-e22}
 \begin{aligned}
L_{j} &= -(j+1)(m-j-1),  \\
M_{j} &= 2(j+1)(m-j-1)(m-2j),\\
Q_{j} &= 4j(j+1)(m-j-1)(m-j).
\end{aligned}   
\end{equation}
We observe that
\begin{align*}
   &(\alpha+1)^2 L_{j}+(\alpha+1)M_{j}  + Q_{j} = - (j+1)(m-j-1)(\alpha + 2j + 1)(\alpha - 2m + 2j + 1),\\
    &\displaystyle{j \choose i-1} = {j+1 \choose i} \dfrac{i}{j+1}, \\
    & D(i, j+1) = D(i-1, j)(m-j-1), \\
   & E(i,j+1)= (\alpha +2i+1)(\alpha +2j+1)E(i+1,j).
\end{align*}
Thus, \eqref{c0-e21} becomes
\begin{eqnarray}\label{c0-e23}
H(i,j) &=& - 2^{i} {j + 1 \choose i} (\alpha -2m +2j +1) D(i, j+1) E(i, j+1).
\end{eqnarray}
It is easy to verify that
\begin{eqnarray*}
    K_{j+1} &=& (\alpha -2m +2j +1) K_{j}, \\
    G(i, j+1) &=& 2^{i} \displaystyle{j+1\choose i} K_{j+1} D(i,j+1) E(i,j+1).
\end{eqnarray*}
Combining the above identities with \eqref{c0-e23}, we see that
$$G(i, j+1) = - K_{j} H(i,j), \,\, \forall i= 1, 2, \cdots, j-1.$$

\paragraph{\textbf{Case 2:}} $i=0$\\
We can rewrite \eqref{c0-e17} as
\begin{equation}\label{c0-e24}
    H(0,j) = D(0, j)E(0,j) B(0,\alpha, \sigma) + 2j D(1,j)E(1,j) C(2, \alpha).
\end{equation}
In addition, 
\begin{align*}
 &B(0, \alpha, \sigma) = -(\alpha +1)(\alpha + 1 - 2m + 4j), \;\; C(2, \alpha) = 2(\alpha + 1) \;\;\text{and} \;\; D(0, j) = 1; \\
 &D(1, j) = m-j, \;\;   E(0,j) = (\alpha +1) E(1, j)\;\; \text{and} \;\; E(0, j+1)=(\alpha + 1 + 2j)E(0, j).  
\end{align*}
Then, \eqref{c0-e24} implies
\begin{equation}\label{c0-e25}
\begin{aligned}
H(0,j) &= E(0,j) \left[ -(\alpha +1)(\alpha +1 - 2m +4j) + 4j (m-j)\right] \\
&= - E(0,j) (\alpha + 1 + 2j)(\alpha +1 -2m +2j)\\
&= -E(0, j+1) (\alpha +1 -2m +2j).
\end{aligned}
\end{equation}
Note that
\begin{eqnarray*}
    K_{j+1} = (\alpha +1 - 2m +2j) K_{j} \quad \text{and} \quad G(0, j+1) = K_{j+1} E(0, j+1).
\end{eqnarray*}
Hence, \eqref{c0-e25} yields
$$G(0, j+1) = - K_{j} H(0,j).$$

\paragraph{\textbf{Case 3:}} $i = j$.\\
We have 
\begin{equation}\label{c0-e26}
    H(j,j) = 2^{j-1}jD(j-1, j)E(j-1,j) A(2j-2, \alpha,\sigma) + 2^{j} D(j,j) E(j, j) B(2j, \alpha,\sigma).
\end{equation}
Also,
\begin{align*}
    & E(j-1, j) = \alpha + 2j -1, \quad E(j, j) = 1 \;\; \text{and}\;\; D(j,j) = (m-1)D(j-1,j),\\
    &A(2j - 2, \alpha,\sigma) = 2(j-1)(\alpha + 2j - 3) + 2(\alpha - 2m + 4j + 1)(3 - m),\\
    &B(2j, \alpha,\sigma) = 4j(\alpha + 2j - 1) - (\alpha -2m +4j +1)(4j + \alpha + 1).
\end{align*}
Hence
\begin{align}
    H(j,j) &= 2^{j} D(j-1,j) \big\{ j(\alpha +2j -1) \big[(j-1)(\alpha + 2j -3) + (\alpha -2m + 4j +1)(3-m) \big] \nonumber \\
    & +  (m-1) \big[ 4j(\alpha +2j -1) - (\alpha-2m +4j +1)(\alpha + 4j+ 1) \big] \big\}. \nonumber
\end{align}
\begin{align*}
    &j(\alpha +2j -1) \big[(j-1)(\alpha + 2j -3) + (\alpha -2m + 4j +1)(3-m) \big] \nonumber \\
    & +  (m-1) \big[ 4j(\alpha +2j -1) - (\alpha-2m +4j +1)(\alpha + 4j+ 1) \big] \big\}
\end{align*}
By arranging the terms, we can write
\begin{equation}\label{c0-e28}
    H(j,j) = 2^{j} D(j-1, j) [(\alpha+1)^2 L_{j} + (\alpha+1)M_{j}+ Q_{j}],
\end{equation}
where $L_j$, $M_j$ and $Q_j$ are defined in \eqref{c0-e22}.
By a simple calculation, we get
\begin{align*}
    & (\alpha+1)^2L_{j}+(\alpha+1)M_{j} +Q_{j}  = - (j+1)(m-j-1)(\alpha + 2j + 1)(\alpha - 2m + 2j + 1),\\
   & D(j, j+1) = (m-j-1) D(j-1, j),\\
  &E(j, j+1) = (\alpha +2j +1),\\
    & K_{j+1} = (\alpha +1 - 2m +2j) K_{j},\\
    & G(j, j+1) = K_{j+1} 2^{j} (j+1) K_{j} D(j, j+1) E(j, j+1).
\end{align*}
Combining the above equalities with \eqref{c0-e28}, we get
$$G(j,j+1) = - K_{j} H(j, j).$$

\paragraph{\textbf{Case 4:}} $i = j+1$.\\
In this case, \eqref{c0-e17} becomes
\begin{equation}\label{c0-e29}
    H(j+1,j) = 2^{j}D(j,j) E(j, j)A(2j, \alpha,\sigma).
\end{equation}
Note that
\begin{eqnarray*} 
        A(2j, \alpha,\sigma) = 2j(2j+ \alpha -1) + 2(\alpha -2m + 4j +1)(1-m) \quad \text{and} \quad E(j, j) =1.
\end{eqnarray*}
Thus,
\begin{equation}\label{c0-ee29}
  \begin{aligned}
    H(j+1,j) & = 2^{j+1} D(j,j)[j(2j+ \alpha -1) + (\alpha -2m + 4j +1)(1-m)]\\
   & = -2^{j+1} D(j,j) (\alpha -2m +2j +1)(m-j-1)
\end{aligned}  
\end{equation}
Also
\begin{eqnarray*}
    D(j+1,j+1) &=& (m-j-1)D(j,j) \quad \text{and} \quad    K_{j+1} = (\alpha +1 - 2m +2j) K_{j};\\
    G(j+1, j+1) &=& 2^{j+1} K_{j+1} D(j+1, j+1).
\end{eqnarray*}
Combining the above equalities with \eqref{c0-ee29}, we get
$$G(j+1, j+1) = - K_{j} H(j+1, j).$$
Finally, from \eqref{c0-e16} and \eqref{c0-e18} follows that \eqref{c0-e11} holds for $j+1$. 
\end{proof}

\begin{cor}\label{c0-corolario2}
Let $u$ be defined by \eqref{c0-e10}. Then, 
\begin{equation}\label{c0-e12}
    (- \Delta_{\alpha})^{m} u = \mathcal{P} \cdot (1+ r^2)^{- \frac{\alpha + 2m + 1 }{2}}. 
\end{equation}
\end{cor}
\begin{proof}
 We choose $j=m$ in the Proposition \ref{c0-prop2}, then
\begin{equation}\label{c0-e30}
    (- \Delta_{\alpha})^{m} u = (1+ r^2)^{- \frac{\alpha + 2m + 1 }{2}} \sum_{i=0}^{m} G(i, m) r^{2i}.
\end{equation}
In addition, 
\begin{equation}\label{c0-e31}
    G(0,m) = K_{m}E(0,m) = \mathcal{P} \quad \text{and} \quad G(i,m) = 0, \,\, \forall i = 1, \cdots, m.
\end{equation}
Thus, from \eqref{c0-e30} and \eqref{c0-e31} follows that \eqref{c0-e12}, as desired.
\end{proof}

\begin{cor}\label{c0-corolario3}
        The function $w$ be defined by 
\begin{equation}\label{c0-e13}
    w(r) = \mathcal{P}^{\frac{\alpha - 2m +1}{4m}} (1+ r^2)^{- \frac{\alpha - 2m +1}{2}}, \quad r>0
\end{equation}
is a positive solution to the equation in \eqref{problema m}.
\end{cor}
\begin{proof}
 We note that 
$$w(r) = \mathcal{P}^{ \frac{\alpha -2m +1}{4m}} u(r) \quad \text{and} \quad 2^{*}-1 = \dfrac{\alpha + 2m +1}{\alpha - 2m +1}.$$
From Corollary \ref{c0-corolario2}, we see
\begin{equation*}
    (- \Delta_{\alpha})^{m} w = \mathcal{P}^{ \frac{\alpha - 2m +1}{4m}} (- \Delta_{\alpha})^{m} u = \mathcal{P}^{ \frac{\alpha +2m +1}{4m}} \cdot (1+ r^2)^{- \frac{\alpha + 2m + 1}{2}} = w^{2^{*}-1}.
\end{equation*}
\end{proof}
\begin{corollary}\label{a0-lema2}
For each $\epsilon>0$, the function $w_{\epsilon}$ be given by \eqref{c0-e13 +}
is a positive solution to the equation in \eqref{problema m}.
\end{corollary}
\begin{proof} Note that we can write $w_{\epsilon}(r) = \epsilon^{-\frac{\alpha+1}{2^{*}}} w(r\epsilon^{-1})$, where  $w$ is given by \eqref{c0-e13}. Thus, by induction on $m$ follows
$$(- \Delta_{\alpha})^{m} w_{\epsilon} (r) = \epsilon^{-(\frac{\alpha+1}{2^{*}} + 2m)} (- \Delta_{\alpha})^{m} w (r \epsilon^{-1}).$$
Thus, from Corollary \ref{c0-corolario3}, we see that
\begin{eqnarray}
(- \Delta_{\alpha})^{m} w_{\epsilon} (r) &=& \epsilon^{-\frac{\alpha+2m+1}{2}}  w^{2^{*}-1} (r\epsilon^{-1}) \nonumber\\
&=& \epsilon^{-\frac{\alpha+1}{2^{*}}(2^{*} - 1)}  w^{2^{*}-1} (r\epsilon^{-1}) \nonumber\\
&=& w_{\epsilon}^{2^{*}-1} (r). \nonumber
\end{eqnarray}
\end{proof}
\subsection{Proof of step 2}\label{subsec1223} Firstly, we note that the each function $w_{\epsilon}$ given by \eqref{c0-e13 +}  have a smooth even extension to  $\mathbb{R}$, i.e., with $w_{\epsilon}(r)=w_{\epsilon}(-r)$. In particular, $w^{(i)}_{\epsilon}(0)=(-1)^{i}w^{(i)}_{\epsilon}(0)$ for all $i=0,1,\cdots, 2m$. Hence, 
\begin{equation}\label{w=0odd}
 w^{(i)}_{\epsilon}(0) =0,\;\; \mbox{for}\;\; i=1,3,\cdots, 2m-1.
\end{equation} 
\begin{lemma}\label{Lemma6Sw}
  Let $w_\epsilon$ be  given by \eqref{c0-e13 +} and let $v\in \mathcal{D}^{m,2}_{\infty}(\alpha)\cap C^{2m}(0,\infty)$ be a positive solution to \eqref{problema m}, then  $g=v-w_{\epsilon}$ is nonoscillatory at $\infty$.
\end{lemma}
\begin{proof}
Let $f_i=(-\Delta_{\alpha})^{\left[ \frac{m}{2} \right]-i}g$ for $i=0,1,\cdots, [m/2]$. By contradiction, suppose that there exists an increasing sequence $(r_n)$ such that $g(r_n)=0$, for $n\in\mathbb{N}$ and  $r_{n}\to \infty$, as $n\to \infty$. From Rolle's lemma, each $f^{\prime}_i$ admits an increasing sequence $(r_{in})$ of zeros such that $r_{in}\to \infty$, for  $i=1,\cdots, [m/2]$. First, we treat the case where $m$
is an even integer. We claim that, for an arbitrary fixed $\delta$, there exists $n_i\in\mathbb{N}$ large enough
\begin{equation}\label{Sw1}
  \begin{aligned}
   |f_i(r)| & \le C_i\delta r^{\frac{4i-(\alpha+1)}{2}},\;\; i=1,\cdots, m/2,\;\;\; \mbox{for}\;\; r>r_{in_{i}}  
\end{aligned}  
\end{equation}
for some positive constant $C_i$. Since $-\Delta_{\alpha}f_1=-r^{-\alpha}(r^{\alpha}f^{\prime}_1(r))^{\prime}=f_0$ and $f^{\prime}_1(r_{1n})=0$, by integrating on $(r_{1n},r)$ we obtain 
\begin{equation}\nonumber
  \begin{aligned}
    r^{\alpha}|f^{\prime}_1(r)|& \le \Big(\int_{r_{1n}}^{r}s^{\alpha}ds\Big)^{\frac{1}{2}}\Big(\int_{r_{1n}}^{r}s^{\alpha}|f_0|^{2}ds\Big)^{\frac{1}{2}}\\
    &\le \Big(\frac{1}{\alpha+1}\Big)^{\frac{1}{2}}r^{\frac{\alpha+1}{2}}\Big(\int_{r_{1n}}^{\infty}s^{\alpha}|\nabla^{m}_{\alpha}g|^{2}ds\Big)^{\frac{1}{2}}\\
    &= \Big(\frac{1}{\alpha+1}\Big)^{\frac{1}{2}}r^{\frac{\alpha+1}{2}}\|\nabla^{m}_{\alpha}g\|_{L^{2}_{\alpha}[r_{1n},\infty)}
\end{aligned}  
\end{equation}
for $r>r_{1n}$.  Since $f_1(\infty)=0$ by Corollary~\ref{UCoroLemmaR} and Proposition~\ref{c0-prop2}, multiplying by $r^{-\alpha}$ and integrating over $[r, \infty)$, we have 
\begin{equation}\nonumber
  \begin{aligned}
   |f_1(r)| & \le \frac{2}{\alpha-3}\Big(\frac{1}{\alpha+1}\Big)^{\frac{1}{2}}\|\nabla^{m}_{\alpha}g\|_{L^{2}_{\alpha}[r_{1n},\infty)} r^{\frac{4-(\alpha+1)}{2}},
\end{aligned}  
\end{equation}
where we also have used that $\alpha+1>2m\ge 4$, if $m$ is even (cf.\eqref{SSS-condi}).
Since $g\in  \mathcal{D}^{m,2}_{\infty}(\alpha)$ and $r_{1n}\to \infty$ as $n\to \infty$, we can choose $n_1$ such that \eqref{Sw1}  for $i=1$. Suppose \eqref{Sw1} holds for some $1\le i\le \frac{m}{2}-1$. Note that
\begin{align}
  -r^{-\alpha}(r^{\alpha}f^{\prime}_{i+1}(r))^{\prime}= -\Delta_{\alpha}f_{i+1}(r)=-\Delta_{\alpha}(-\Delta_{\alpha})^{\frac{m}{2}-(i+1)}g(r)=f_i(r).
\end{align}
Thus, by choosing $n_{i+1}$ large enough to $r_{(i+1)n_{i+1}}>r_{in_i}$ and by integrating on $(r_{(i+1)n},r)$  we obtain  
\begin{equation}
  \begin{aligned}
    r^{\alpha}|f^{\prime}_{i+1}(r)|& =\int_{r_{(i+1)n}}^{r}s^{\alpha}|f_i(s)|ds\le  C\delta \int_{r_{(i+1)n}}^{r}s^{\alpha+\frac{4i-(\alpha+1)}{2}}ds \\
    &\le C\delta \int_{0}^{r}s^{\frac{\alpha+4i-1}{2}}ds\\
    &=\frac{ 2C\delta }{\alpha+4i+1}r^{\frac{\alpha+4i+1}{2}}
\end{aligned}  
\end{equation}
for $r>r_{(i+1)n_{i+1}}$. By Corollary~\ref{UCoroLemmaR} and Proposition~\ref{c0-prop2}, we are able to integrate again to obtain 
\begin{equation}\label{SwHI}
  \begin{aligned}
    |f_{i+1}(r)|&\le \frac{ 2C\delta }{\alpha+4i+1} \int_{r}^{\infty}s^{\frac{4i-\alpha+1}{2}}ds\\
    &=\frac{2C\delta }{\alpha+4i+1}\frac{2}{\alpha- 4i-3}r^{\frac{4(i+1)-(\alpha+1)}{2}}
\end{aligned}  
\end{equation}
where we have used that $4i+3-\alpha\le 4(\frac{m}{2}-1)+3-\alpha=2m -4+3-\alpha=2m-\alpha-1<0$. Thus, \eqref{SwHI}
ensures \eqref{Sw1} for $i+1$. Now, by setting $$n=\max\{n_1, \cdots, n_{\frac{m}{2}}\}\;\; \mbox{and}\;\; \rho_n=\max\{r_{1n_1},\cdots, r_{\frac{m}{2}n_{\frac{m}{2}}} \}$$
we obtain 
\begin{equation}\label{Sw2}
  \begin{aligned}
   |f_i(r)| & \le C\delta r^{-\frac{\alpha-4i+1}{2}},\;\;\; i=1,\cdots, m/2,\;\;\; \mbox{for}\;\; r>\rho_{n}  
\end{aligned}  
\end{equation}
where  $C=\max\{C_1, \cdots, C_{\frac{m}{2}}\}$.

By applying the mean value theorem and Corollary~\ref{a0-lema2}, we deduce that $g$ satisfies the  linear equation
\begin{equation}\label{SWLinear}
 (-\Delta_{\alpha})^{m}g=\xi(r)g,\;\;\; r>0   
\end{equation}
where $\xi(r)=(2^{*}-1)\big[\theta v+(1-\theta)w_{\epsilon}\big]^{{2^{*}-2}}$ for some $0\le \theta=\theta(r)<1$. From \eqref{c0-e13 +} and \eqref{Sw2}, there exist positive constants $C_0$ and $\rho$ such that
\begin{equation}\nonumber
   \left\{\begin{aligned}
        & w_{\eps}(r)\le \frac{C_0\delta}{2}r^{-\frac{\alpha-2m+1}{2}}\\
        &v(r)\le C_0\delta r^{-\frac{\alpha-2m+1}{2}} \\
        &\xi(r)\le (2^*-1)(C_0\delta)^{2^{*}-2}r^{-2m}
    \end{aligned}\right.,\quad \mbox{for}\;\; r>\rho,
\end{equation}
where we have used that $2^*-2=4m/(\alpha-2m+1)$. Since $\delta>0$ is arbitrary we have
\begin{equation}\label{Sw3}
    \lim_{r\to\infty} r^{2m}\xi(r)=0.
\end{equation}
Now, consider the Cauchy-Euler  type equation  (cf. \cite[Lemma 2.9]{JN-JF})
\begin{equation}\label{SWEuler}
 (-\Delta_{\alpha})^{m}y-\sigma r^{-2m}y=0,\;\;\;\mbox{for}\;\;  r>0.   
\end{equation}
Since $\Delta_{\alpha}(r^{\lambda})=\lambda(\lambda+\alpha-1)r^{\lambda-2}$ for $\lambda\ge 1$, the characteristic polynomial of \eqref{SWEuler} is given by
\begin{equation}\label{Cauchy-Euler-Poly}
P(\lambda)=(-1)^m\prod_{j=1}^{m}(\lambda-2j+2)(\lambda-2j+\alpha+1)-\sigma   
\end{equation}
which has $2m$ distinct real roots for $\sigma>0$ sufficiently small and $\alpha-2m+1>0$. Then, the solutions of \eqref{SWEuler} are nonoscillatory at $\infty$, if $\sigma$ is small enough. From \eqref{Sw3}, we have $\xi(r)<\sigma r^{2m}$ for $r$ large enough and from the comparison result in \cite[Corollary~1]{zbMATH03572527} we conclude that \eqref{SWLinear} is also nonoscillatory. This contradiction implies that our result holds for $m$ even. An analogous argument allows us to conclude the result for odd $m$ as well.
\end{proof}
Let $v\in \mathcal{D}^{m,2}_{\infty}(\alpha)\cap C^{2m}(0,\infty)$ be a positive $(2m-1)$-th nonsingular at $r=0$ solution to \eqref{problema m} and such that $v^{(i)}(0) = 0$, for  $i = 1, 3, \ldots, 2m-1$. As in Lemma~\ref{Lemma6Sw}, set $g=v-w_{\epsilon}$, where  $w_{\epsilon}$ given by \eqref{c0-e13 +}. 
Then, by setting
\begin{equation}\nonumber
    \left\{\begin{aligned}
        &g_j=(-\Delta_{\alpha})^{j}g,\;\;\; j=0,1,\cdots m-1\\
       &f=v^{2^{*}-1}-w^{2^{*}-1}_{\epsilon} 
    \end{aligned}\right.
\end{equation}
we have that $g_0,g_1\cdots, g_{m-1}$ solves the system
\begin{equation}\label{SwSystem}
\left\{\begin{aligned}
&\left.
\begin{aligned}
&-\Delta_{\alpha} g_{j} = g_{j+1}, && i=0,1,\cdots,m-2\\
&-\Delta_{\alpha} g_{m-1} = f
\end{aligned}
\right\}
\quad \text{in } (0,\infty),
\\
& g_0 \in \mathcal{D}^{m,2}_{\infty}(\alpha),\qquad 
  g_0^{(i)}(0)=0,\quad i=1,3,\cdots,2m-1
\end{aligned}\right.
\end{equation}
where we also have used \eqref{w=0odd}.

\begin{remark} Note that Step~2 follows if  $g_0=v-w_{\epsilon}\equiv 0$ for some $\epsilon>0$.
\end{remark}
Firstly, we will prove the following:
\begin{lemma}\label{LemmaSw7}
  Suppose that $v(0)=w_{\epsilon}(0)$. Then,   neither $g_0$ nor $g_1$ has any zeros in $(0,\infty)$, unless $g_0\equiv0$.
\end{lemma}
\begin{proof}
It is well-known from the theory of ordinary differential equations that the set of zeros of $g_0$ in $(0,\infty)$ cannot have an accumulation point. Hence, by the Lemma~\ref{Lemma6Sw}, $g_0$ has only finitely many zeros in $(0,\infty)$. Let us say that there are exactly $k$ zeros of $g_0$ in $(0,\infty)$, and denote them by $0<r_1<r_2<\ldots<r_k$ and set $r_0=0$. Note that $g_0(0)=0$, since we are assuming $v(0)=w_{\epsilon}(0)$. From Lemma~\ref{lemma-iterationRev}, we get $v^{\prime}(0)=0$. This together with  \eqref{w=0odd} yield $g^{\prime}_0(0)=0$  and, from Corollary~\ref{UCoroLemmaR} we obtain $g_0(\infty)=0$. Thus, $g_0$ has at least $k+2$ extremum points $s_0<s_1< \cdots <s_k<s_{k+1}$ with $s_0=0$, $s_{k+1}>r_k$, and $s_{\ell}\in (r_{\ell-1}, r_{\ell})$ for $\ell=1,\dots, k$. From $\Delta_{\alpha}g_0=-g_1$, we get $(r^{\alpha}g^{\prime}_0)^\prime=-r^{\alpha}g_1$. Using $g^{\prime}_0(s_{\ell})=0$ and integrating over $(s_{\ell},t)$,   we get 
\begin{equation}\label{passS1}
    g^{\prime}_0(t)=-\frac{1}{t^{\alpha}}\int_{s_{\ell}}^{t}s^{\alpha}g_1(s)ds,\;\; \forall\;t>s_{\ell},\;\; \;\ell=0,1, \cdots, k+1.
\end{equation}
Taking $\ell=0$ and $t=s_1$, and using $g_0'(s_1)=0$, we deduce that
$
\int_{0}^{s_1} s^{\alpha} g_1(s)\,ds=0.
$
Therefore,  either $g_1$
changes sign on $(0,s_1)$ or  $g_1\equiv0$ on $(0, s_1)$. If $g_1 \equiv 0$ on $(0,s_1)$, then $g_1=-\Delta_{\alpha} g_0 \equiv 0 $ and so $(r^{\alpha} g^{\prime}_0)'=0$ on $(0,s_1).$
Hence, $g_0'(s_1)=0$ yields $g^{\prime}_0\equiv0$ on $(0,s_1)$
and therefore $g_0(0)=0$ implies $g_0 \equiv 0 $ on $(0,s_1)$ contradicting the fact that the zeros of $g_0$ have no accumulation points.
Hence, $g_1$ changes sign on $(0, s_1)$. Analogously, from \eqref{passS1} with $\ell=1$ and $t=s_2>s_1$, and using $g_0^{\prime}(s_2)=0$ it follows that either $g_1$ changes sign on $(s_1,s_2)$ or $g_1 \equiv 0$ on $(s_1,s_2)$. Since $r_1\in (s_1, s_2)$ and $g_0(r_1)=0$, the second alternative also implies $g_0 \equiv 0$ on $(s_1,s_2)$ which is impossible. By iterating this argument, we observe that $g_1$ admits successive sign changes on each interval $(s_{\ell-1}, s_\ell)$, $\ell = 1, \ldots, k+1$. So, $g_1$ has at least $k+1$ zeros in $(0,\infty)$. 

From Corollary~\ref{UCoroLemmaR}, Corollary~\ref{corollaryreferee} and  Proposition~\ref{c0-prop2}, we get $g^{\prime}_j(0)=0$ and $g_j(\infty)=0$ for $j=1, \cdots m-1$. Hence,  the same argument ensures that each $g_2, g_3, \cdots, g_{m-1}, \Delta_{\alpha}g_{m-1}=-f$ have at least $k+1$ positive zeros. Since $f=v^{2^{*}-1}-w^{2^{*}-1}_{\epsilon}$ has exactly the same $k$ positive zeros of $g_0$, we have a contradiction. Hence, $k=0$ and $g_0$ does not change sign in $(0,\infty)$. Now, if $g_1(r_1)=0$ for some $r_1>0$, then $g_1$ has at least two extremum points $0=s_0<r_1<s_1$. Since $g^{\prime}_1(0)=0$, by integrating $(r^{\alpha}g^{\prime}_1)^{\prime}=-r^{\alpha}g_2$ we obtain 
\begin{equation}\label{G1passS1}
    g^{\prime}_1(t)=-\frac{1}{t^{\alpha}}\int_{0}^{t}s^{\alpha}g_2(s)ds,\;\; \forall\;t>0.
\end{equation}
Since $g^{\prime}_1(s_1)=0$ and  $g_1(r_1)=0$, as before, we conclude that $g_2$ changes sign on $(0, s_1)$. By repeating the argument of the first part we are able to show that $f$ has a positive zero, and then $g_0$ would have a positive zero.
\end{proof}
\begin{lemma}\label{lemmabrabo}
Let $g=v-w_{\epsilon}$ under the assumptions of Lemma~\ref{LemmaSw7}. If $g^{(2k)}(0)=0$ for $k=1,\cdots,K$ with $K\le m-1$, then
$$(-\Delta_{\alpha})^{K+1}g(0) = C_{\alpha,K}(-1)^{K+1}g^{(2K+2)}(0)$$
for some constant $C_{\alpha,K}>0$.
\end{lemma}
\begin{proof}
First, we note that \eqref{Rw=v0}, the assumption $v^{(j)}(0)=0$ for $j=1, 3, \cdots, 2m-1$ and \eqref{w=0odd} yield
\begin{equation}\label{G0odd}
    g^{(j)}(0)=0,\;\; \mbox{for} \;\; j=1,3,\cdots, 2m-1.
\end{equation}
This together with  the assumption $g^{(2k)}(0)=0$, for $k=1, \cdots, K$ yield
\begin{equation}\label{G0oddFUL}
    g^{(j)}(0)=0,\;\; \mbox{for} \;\; j=1,2,3,\cdots, 2K+1.
\end{equation}
Let us denote 
\begin{equation}\label{RTaylor}
    R(r)=g(r)-\frac{g^{(2K+2)}(0)}{(2K+2)!}r^{2K+2}.
\end{equation}
From the L'Hospital  rule together with   \eqref{G0oddFUL}, we get 
\begin{equation}\nonumber
  \lim_{r\to 0}\frac{g^{(1)}(r)}{r^{2K+1}}=\lim_{r\to 0}\frac{g^{(2)}(r)}{(2K+1)r^{2K}}=\cdots =\lim_{r\to 0}\frac{g^{(2K+1)}(r)}{(2K+1)! r}=\frac{g^{(2K+2)}(0)}{(2K+1)!}.
\end{equation}
 Then, 
\begin{equation}
 \lim_{r\to 0}\frac{R^{(1)}(r)}{r^{2K+1}}=\lim_{r\to 0}\Big[\frac{g^{(1)}(r)}{r^{2K+1}}-\frac{g^{(2K+2)}(0)}{(2K+1)!}\Big]=0. 
\end{equation}
Analogously 
\begin{equation}
 \lim_{r\to 0}\frac{R^{(2)}(r)}{r^{2K}}=\lim_{r\to 0}\Big[\frac{g^{(2)}(r)}{r^{2K}}-\frac{g^{(2K+2)}(0)}{(2K)!}\Big]=0. 
\end{equation}
Generally,
\begin{equation}\label{RGeralFull}
 \lim_{r\to 0}\frac{R^{(k)}(r)}{r^{2K+2-k}}=\lim_{r\to 0}\Big[\frac{g^{(k)}(r)}{r^{2K+2-k}}-\frac{g^{(2K+2)}(0)}{(2K+2-k)!}\Big]=0,\;\;\mbox{for}\;\; k=1,2,  \cdots, 2K+1.
\end{equation}
By using the expansion in \cite[Lemma 2.9]{JN-JF}, we obtain
\begin{equation}\nonumber
    (-\Delta_{\alpha})^{K+1}R=(-1)^{K+1}\Big[R^{(2K+2)}+c_1\frac{R^{(2K+1)}}{r}+c_2\frac{R^{(2K)}}{r^2}+\cdots+c_{2K}\frac{R^{(2)}}{r^{2K}}+c_{2K+1}\frac{R^{(1)}}{r^{2K+1}}\Big].
\end{equation}
Noticing that $R^{(2K+2)}(0)=0$, from \eqref{RGeralFull}
\begin{equation}\label{DR0Full}
   (-\Delta_{\alpha})^{K+1}R(0)=\lim_{r\to 0} (-\Delta_{\alpha})^{K+1}R=0.
\end{equation}
Recalling $\Delta_{\alpha}(r^{\lambda})=\lambda(\lambda+\alpha-1)r^{\lambda-2}$ for $\lambda\ge 1$, from \eqref{RTaylor} we have
\begin{align*}
(-\Delta_{\alpha})^{K+1}g&=(-1)^{K+1}\Delta^{K+1}_{\alpha}\Big[R+\frac{g^{(2K+2)}(0)}{(2K+2)!}r^{2K+2}\Big]\\
&=(-\Delta_{\alpha})^{K+1}R+C_{\alpha,K}(-1)^{K+1}g^{(2K+2)}(0),
\end{align*}
where
$$C_{\alpha,K}=\frac{1}{(2K+2)!}\prod_{\ell=1}^{K+1}(2K+2-2\ell+2)(2K+2-2\ell+\alpha+1)>0,\;\;\; \mbox{for}\;\; \alpha-2m+1>0.$$
Taking into account \eqref{DR0Full}, letting $r\to 0$
\begin{equation}\label{Df_0=CFull}
  (-\Delta_{\alpha})^{K+1}g(0)=C_{\alpha,K}(-1)^{K+1}g^{(2K+2)}(0).
\end{equation} 
\end{proof}
We are now ready to prove Step~2, which is stated in the following result:
\begin{lemma}\label{lema m} 
If $v\in \mathcal{D}^{m,2}_{\infty}(\alpha)\cap C^{2m}(0,\infty)$ is a positive $(2m-1)$-th nonsingular at $r=0$ solution to \eqref{problema m} such that $v^{(i)}(0) = 0$, for  $i = 1, 3, \ldots, 2m-1$, then  $v=w_{\epsilon}$ for some $\epsilon>0$, where \( w_{\epsilon} \) is given by \eqref{c0-e13 +}.
\end{lemma}
\begin{proof}
    From Corollary \ref{a0-lema2}, for each $\epsilon>0$ the function  $w_{\epsilon}$ solves \eqref{problema m} with $$w_{\epsilon}(0)=\mathcal{P}^{\frac{\alpha-2m+1}{4m}}\epsilon^{-\frac{\alpha-2m+1}{2}}.$$ From $v(0)>0$, we can choose $\Bar{\epsilon}>0$ such that 
\begin{equation}\label{Rw=v0}
    w_{\bar{\epsilon}}(0) = v(0).
\end{equation}
Setting $v_j=(-\Delta_{\alpha})^{j}v$, for $j=0, 1, \cdots, m$, by Corollary~\ref{corollaryreferee}, we conclude that  $v_0, v_1, \cdots, v_{m-1}$  satisfy the system
\begin{equation}\label{PVISw}
\left\{\begin{aligned}
&\left.
\begin{aligned}
&u_0=u\\
&u_j=-\Delta_{\alpha}u_{j-1}, \quad j=1,\cdots,m-1\\
&-\Delta_{\alpha}u_{m-1}= |u|^{2^{*}-2}u\\
\end{aligned}
\right\}
\quad \text{in } (0,\infty),
\\
& u_j(0)=v_j(0)\;\; \mbox{and}\;\;  u^{\prime}_j(0)=0,\quad j=0,1,\cdots,m-1.
\end{aligned}\right.
\end{equation} 

Due to the singularity of the term $\frac{\alpha}{r}$ at the origin ($r=0$), the uniqueness of solution to the initial value problem \eqref{PVISw} cannot be derived directly from the classical Picard-Lindel\"{o}f theorem. Nevertheless, using a Volterra integral approach \cite[Theorem~XIII, p.71]{Walter}, we prove below, in Corollary~\ref{v=w}, that \eqref{PVISw} admits at most one solution $u\in C^{2m}(0,\infty)$. 

In the same way of Lemma~\ref{Lemma6Sw} and Lemma~\ref{LemmaSw7}, we set
\begin{equation}\label{gif}
    \left\{\begin{aligned}
    &g=v-w_{\bar{\epsilon}}\\
     &g_j=(-\Delta_{\alpha})^{j}g,\;\;\; j=0,1,\cdots m-1\\
       &f=v^{2^{*}-1}-w^{2^{*}-1}_{\bar{\epsilon}}. 
    \end{aligned}\right.
\end{equation}
If $g_j(0)=0$ for $j=1,\cdots, m-1$, in view of \eqref{Rw=v0}, both $v$ and $w_{\bar{\epsilon}}$ satisfy the initial value problem \eqref{PVISw} and, by uniqueness proved in Corollary~\ref{v=w}, we obtain $v=w_{\bar{\epsilon}}$. By contradiction, suppose that  there exists 
\begin{equation}\label{JcontraSw-Hfalse}
1\le J\le m-1\;\; :\;\;g_{J}(0)>0\;\mbox{and} \;\; g_{j}(0) =0,\;\;\mbox{for}\;\; 1\le j<J.  
\end{equation}
Let $f_0=(-1)^{J}g_0$. We claim that 
\begin{equation}\label{ClaimFinal}
    \begin{aligned}
 f_0^{(j)}(0)=0\;\;\mbox{for}\;\; 1\le j<2J \;\;\mbox{and}\;\;  f^{(2J)}_0(0)>0.
    \end{aligned}
\end{equation}
Since $v^{(j)}(0)=0$ for $j=1, 3, \cdots, 2m-1$ and \eqref{w=0odd} hold, we already have
\begin{equation}\label{f0odd}
    f^{(j)}_0(0)=(-1)^{J}\big[v^{(j)}(0)-w^{(j)}_{\bar\epsilon}(0)\big]=0,\;\; \mbox{for} \;\; j=0,1,3,\cdots, 2m-1.
\end{equation}
If $J=1$, \eqref{f0odd} ensures $f^{\prime}_0(0)=0$ and  by the L'Hospital rule, we have  
$$f^{(2)}_0(0)=-g^{(2)}_0(0)=-\frac{1}{\alpha+1}\lim_{r\to 0}\Big(g^{(2)}_0(r)+\frac{\alpha}{r}g^{(1)}_0(r)\Big)=\frac{1}{\alpha+1}(-\Delta_{\alpha}g_0)=\frac{1}{\alpha+1}g_1(0)>0.$$
Thus, we can assume $J>1$. In order to show \eqref{ClaimFinal},  we first prove that for $J$ as in \eqref{JcontraSw-Hfalse}, we have 
\begin{equation}\label{CGFinal}
    \begin{aligned}
f_0^{(2j)}(0)=(-1)^{J}g^{(2j)}_0(0)=0,\;\;\mbox{for}\;\; 1\le j<J.
    \end{aligned}
\end{equation}
We proceed by induction. By the L'Hospital rule again together with \eqref{f0odd}, we have  
$$0=g_{1}(0)=-\Delta_{\alpha}g_0(0)=-\lim_{r\to 0}\Big(g^{(2)}_0(r)+\frac{\alpha}{r}g^{(1)}_0(r)\Big)=-(1+\alpha)g^{(2)}_0(0)$$
which gives \eqref{CGFinal} for $j=1$. Suppose that $g^{(2j)}(0)=0$ for $1\le j<J-1$. Then,  from Lemma~\ref{lemmabrabo} and the definition of $J$, we have 
$$0=g_{j+1}(0)=(-\Delta_{\alpha})^{j+1}g_0(0) = C_{\alpha,j}(-1)^{j+1}g^{(2j+2)}_0(0)$$
for some constant $C_{\alpha,j}>0$. Thus, $g^{(2j+2)}_0(0)=0$. Therefore, $g_0^{(2j)}(0)=0$ for $j=1, \cdots, J-1$. In view of \eqref{CGFinal}, from Lemma~\ref{lemmabrabo} we also have 
\begin{equation}\label{2JP}
    C_{\alpha,J}f^{(2J)}_0(0)=C_{\alpha,J}(-1)^{J}g^{(2J)}_0(0)=(-\Delta_{\alpha})^{J}g_0(0) = g_{J}(0)>0,
\end{equation}
with $C_{\alpha,J}>0$. From \eqref{f0odd}, \eqref{CGFinal} and \eqref{2JP}, we conclude  that \eqref{ClaimFinal} holds. 
Now, from \eqref{ClaimFinal} and Taylor's expansion, we have that $f_0$ has a local minimum at $r=0$ with $f_0(0)=0$. Thus, by Lemma~\ref{LemmaSw7} we have $f_0>0$  on $(0,\infty)$.  Analogously, let $f_1=-\Delta_{\alpha}f_0$. Note that 
$$
\Delta^{j}_{\alpha}f_1=-\Delta^{j+1}_{\alpha}f_0=(-1)^{J-1}\Delta^{j+1}_{\alpha}g_0=(-1)^{J-1-(j+1)}(-\Delta_{\alpha})^{j+1}g_0=(-1)^{J-j-2}g_{j+1}.
$$
In particular, by definition of $J$
\begin{equation}
    \Delta^{j}_{\alpha}f_1(0)=0,\;\; \mbox{for all}\;\; j=0,1,\cdots, J-2\;\;\mbox{and}\;\; \Delta^{J-1}_{\alpha}f_1(0)=-g_{J}(0)<0.
\end{equation}
By repeating the above argument with $f_1$ instead of $f_0$, we can see that 
\begin{equation}\label{F1ClaimFinal}
    \begin{aligned}
 f_1^{(j)}(0)=0\;\;\mbox{for}\;\; 1\le j<2J-2 \;\;\mbox{and}\;\;  f^{(2J-2)}_1(0)<0.
    \end{aligned}
\end{equation}
Hence, $f_1$ has a local maximum at $r=0$ with $f_1(0)=0$, and by Lemma~\ref{LemmaSw7} we obtain $f_1 < 0$ on $(0, \infty)$. On the other hand, $f_1=-\Delta_{\alpha}f_0=-r^{-\alpha}(r^{\alpha}f^{\prime}_0(r))^{\prime}$ and $f^{\prime}_0(0)=0$ imply
\begin{equation}\nonumber
    f_0^{\prime}(r)=-\frac{1}{r^{\alpha}}\int_{0}^{r}s^{\alpha}f_1(s)ds>0,\;\; \mbox{for all}\;\; r>0.
\end{equation}
Thus, $f_0$ is a positive and strictly increasing function on $(0,\infty)$ with $f_0(0)=0$ and $f_{0}(\infty)=0$, by Corollary~\ref{UCoroLemmaR} (or Lemma~\ref{LR2}). This is a contradiction. Therefore, our uniqueness result is proven.
\end{proof}
\subsection{Proof of Corollary \ref{cor 1.2}}
Suppose that $z \in \mathcal{D}^{m,2}_{\infty}(\alpha)$ is a minimizer for  $\mathcal{S}(\alpha,m)$ such that $z$ is $(2m-1)$-th nonsingular at $r=0$ and $z^{(i)}(0) = 0$,  for   $i = 1, 3, \ldots, 2m-1$. By Theorem \ref{Existencia de minimizante suave}, we have $z \in C^{2m}(0, \infty)$ and solves the equation \eqref{problema m}. Hence,  Lemma \ref{sinal de minimizante} ensures that $z$ has a defined sign, so we choose $z>0$. By Theorem \ref{Teorema unicidade} it follows that $z=w_{\Bar{\epsilon}}$, for some $\Bar{\epsilon}>0$. Conversely,   from \eqref{cond min non-singular} and Lemma~\ref{sinal de minimizante} we are able to choose a positive minimizer $u_0$ for $\mathcal{S}(\alpha,m)$ such that it is $(2m-1)$-th nonsingular at $r=0$,  $u^{(i)}_0(0) = 0$,  for   $i = 1, 3, \ldots, 2m-1$  and $\|u_0\|_{L^{2^{*}}_{\alpha}}=1$. The Lagrange multiplier theorem yields 
\begin{equation}\label{a0-e20-novo}
    \int_{0}^{\infty} \nabla^{m}_{\alpha} u_{0} \nabla^{m}_{\alpha} v r^{\alpha}\, dr = \mathcal{S}\int_{0}^{\infty} |u_{0}|^{2^{*}-2} u_{0} v r^{\alpha}\, dr, \quad \mbox{for all}\;\; v \in \mathcal{D}^{m,2}_{\infty}(\alpha).
\end{equation}
In particular, the function $ u_{\mathcal{S}}= \mathcal{S}^{\frac{1}{2^{*} - 2}} u_{0}$ satisfies  the equation
\begin{equation*}
    (-\Delta_{\alpha})^{m} u_{\mathcal{S}} = |u_{\mathcal{S}}|^{2^{*}-2} u_{\mathcal{S}}\;\;\mbox{in}\;\; (0, \infty).
\end{equation*}
Using the Theorem~\ref{Teorema unicidade}, we obtain that there is $\epsilon_0>0$ such that $w_{\epsilon_{0}}=u_{\mathcal{S}}=\mathcal{S}^{\frac{1}{2^{*} - 2}} u_{0}$. Thus, 
\begin{equation}\label{touch}
    \frac{\|\nabla^{m}_{\alpha} w_{\epsilon_0}\|_{L^{2}_{\alpha}}}{\|w_{\epsilon_0}\|_{L^{2^*}_{\alpha}}}=\|\nabla^{m}_{\alpha}u_0\|_{L^{2}_{\alpha}}=\mathcal{S}^{\frac{1}{2}}.
\end{equation}
It follows that $w_{\epsilon_0}$ is a minimizer. Now, we note that  
\begin{equation}\label {wuu}
    w_{\epsilon}(r)=\mathcal{P}^{\frac{\alpha - 2m +1}{4m}} \left( \dfrac{\epsilon}{\epsilon^{2}+ r^2} \right)^{\frac{\alpha - 2m +1}{2}}=\mathcal{P}^{\frac{\alpha - 2m +1}{4m}} \left[\epsilon^{-\frac{\alpha+1}{2^{*}}}u\big(\frac{r}{\epsilon}\big)\right],
\end{equation}
where $u$ is defined by \eqref{c0-e10}. Thus, setting $s=r/\epsilon$ we get
\begin{align}\label{lp-naoe}
\|w_{\epsilon}\|_{L^{2^*}_{\alpha}}=\mathcal{P}^{\frac{\alpha - 2m +1}{4m}}\|u\|_{L^{2^*}_{\alpha}},\quad \mbox{for all}\;\; \epsilon>0.
\end{align}
Further, from \cite[Lemma~3.1]{JN-JF}, for any $j\in\mathbb N$ 
\begin{equation}\label{c0-e61-par}
    (-\Delta_{\alpha})^{j}\left[\epsilon^{-\frac{\alpha+1}{2^{*}}}u\big(\frac{r}{\epsilon}\big)\right] = \epsilon^{-(\frac{\alpha+1}{2^*} + 2j)} [(-\Delta_{\alpha})^{j} u ](\frac{r}{\epsilon})
\end{equation}
and 
\begin{equation}\label{c0-e61-impar}
    \left((-\Delta_{\alpha})^{j}\left[\epsilon^{-\frac{\alpha+1}{2^{*}}}u\big(\frac{r}{\epsilon}\big)\right]\right)^{\prime} = \epsilon^{-(\frac{\alpha+1}{2^*} + 2j+1)} [(-\Delta_{\alpha})^{j} u ]^{\prime}(\frac{r}{\epsilon}).
\end{equation}
Setting $s=r/\epsilon$ and using \eqref{wuu}, \eqref{c0-e61-impar} and \eqref{c0-e61-impar}, we obtain 
\begin{align}\label{gradnaoe}
   \|\nabla^{m}_{\alpha} w_{\epsilon}\|_{L^{2}_{\alpha}}=\mathcal{P}^{\frac{\alpha - 2m +1}{4m}}\|\nabla^{m}_{\alpha} u\|_{L^{2}_{\alpha}},\;\; \mbox{for all}\;\; \epsilon>0.
\end{align}
From \eqref{touch}, \eqref{lp-naoe} and \eqref{gradnaoe} it follows that 
\begin{equation}\nonumber
\frac{\|\nabla^{m}_{\alpha} w_{\epsilon}\|_{L^{2}_{\alpha}}}{\|w_{\epsilon}\|_{L^{2^*}_{\alpha}}}=    \frac{\|\nabla^{m}_{\alpha} w_{\epsilon_0}\|_{L^{2}_{\alpha}}}{\|w_{\epsilon_0}\|_{L^{2^*}_{\alpha}}}=\mathcal{S}^{\frac{1}{2}}, \quad\mbox{for all}\;\; \epsilon>0.
\end{equation}    

\section{The best constant $\mathcal{S}^{-\frac{1}{2}}(\alpha,m)$}
\label{sec4}
\begin{proof}[Proof of Theorem \ref{Theorem 1.3}]
Let  $z>0$ be a minimizer for  $\mathcal{S}(\alpha, m)$, as in the assumption \eqref{cond min non-singular}. Since the dilation of a minimizer is still a minimizer, we choose $z= w_{1}$ to calculate the value of $\mathcal{S}$, where $w_{1}$ is given in \eqref{c0-e13 +}. Since $w_{1}$ solves \eqref{problema m}, we can write
$$\| \nabla_{\alpha}^{m} z\|^{2}_{L^{2}_{\alpha}} = \|z\|^{2^{*}}_{L_{\alpha}^{2^{*}}}.$$
Thus,
\begin{equation}\label{S-integral}
 \mathcal{S} = \dfrac{\|\nabla^{m}_{\alpha} z \|^{2}_{L^{2}_{\alpha}}}{\|z\|^{2}_{L^{2^{*}}_{\alpha}}} = \|z\|^{2^{*}-2}_{L^{2^{*}}_{\alpha}} = \mathcal{P} \left[\int_{0}^{\infty} \frac{r^{\alpha}}{(1+r^{2})^{\alpha+1}}\, dr\right]^{\frac{2m}{\alpha+1}}.   
\end{equation}
We recall the identities 
\begin{equation}\nonumber
    \int_{0}^{\infty}\frac{s^{x-1}}{(1+s)^{x+y}}ds=\frac{\Gamma(x)\Gamma(y)}{\Gamma(x+y)}, \;\;\Gamma(1)=1\;\;\mbox{and}\;\; \Gamma(x+1)=x\Gamma(x),
\end{equation}
see for instance \cite{S-functions}. Hence, 
$$ \int_{0}^{\infty} \frac{r^{\alpha}}{(1+r^{2})^{\alpha+1}}\, dr =\frac{\Gamma(\frac{\alpha+1}{2})\Gamma(\frac{\alpha+1}{2})}{2\Gamma(\alpha+1)}$$
which together \eqref{S-integral} yields
$$
\mathcal{S}^{-\frac{1}{2}} =  \mathcal{P}^{-\frac{1}{2}}\left[\dfrac{2\Gamma(\alpha+1)}{\Gamma(\frac{\alpha+1}{2})\Gamma(\frac{\alpha+1}{2})}\right]^{\frac{m}{\alpha+1}}
$$
as desired.
\end{proof}
\begin{proof}[Proof of Corollary~\ref{corom=2}]
It is a direct consequence of Lemma~\ref{lemma-iterationRev} and Theorem~\ref{Theorem 1.3}.
\end{proof}
\appendix
\numberwithin{equation}{section}
\section{Uniqueness}
Let $m\ge 1$ be an integer and let $\alpha\in\mathbb{R}$ satisfy \eqref{SSS-condi}. 
Denote by $\Delta_{\alpha}$ the $\alpha$-generalized radial Laplacian defined in \eqref{laplace-geral}. For $a=(a_0, \dots, a_{m-1})\in\mathbb{R}^m$, let us consider the initial value problem
\begin{equation}\label{PVILocal}
\left\{\begin{aligned}
&\left.
\begin{aligned}
&u_0=u\\
&u_j=-\Delta_{\alpha}u_{j-1}, \quad j=1,\cdots,m-1\\
&-\Delta_{\alpha}u_{m-1}= |u|^{2^{*}-2}u\\
\end{aligned}
\right\}
\quad \text{in } (0,\delta], \; \delta>0
\\
& u_j(0)=a_j\;\; \mbox{and}\;\;  u^{\prime}_j(0)=0,\quad j=0,1,\cdots,m-1.
\end{aligned}\right.
\end{equation} 
We observe that the operator $\Delta_{\alpha}u = u^{\prime\prime} + \frac{\alpha}{r}u^{\prime}$ contains a singularity at $r=0$. Hence,  the initial conditions in \eqref{PVILocal} are interpreted as right limits at the origin, that is,
\begin{equation}\label{l-IC}
   u_j(0)=\lim_{r\to 0^+} u_j(r) \quad\mbox{and}\quad
u_j'(0)=\lim_{r\to 0^+} u_j'(r). 
\end{equation}
Due to this singularity, we cannot directly apply the classical Picard-Lindel\"{o}f theorem to obtain uniqueness of solutions for \eqref{PVILocal}.  We employ an approach based on Volterra integral equations, as in \cite[Theorem~XIII, p.71]{Walter} to ensure existence and uniqueness, at least for $\delta>0$ small enough. In fact, for  fixed $d>0$ we choose $\delta>0$ such that 
\begin{equation}\label{delta-peq}
   \delta^2=\min\left\{\frac{d}{d+|a|},\frac{d}{(d+|a|)^{2^*-1}},\frac{1}{(2^*-1)(2(d+|a|))^{2^*-2}}\right\},
\end{equation}
where $|a|=\max_{0\le j\le m-1} |a_j|$.
\begin{theorem}\label{thm-uniq} 
If $\delta>0$ is given by \eqref{delta-peq}, then the system \eqref{PVILocal} has exactly one solution $(u_0,\dots, u_{m-1})$ such that $|u_j(r)-a_j|\le d$, $j=0, \cdots, m-1$ for  $r\in I_{\delta}=[0, \delta]$.
\end{theorem}
\begin{proof}
Let $u_m= |u_0|^{2^*-2}u_0$. By noticing that
$
u_j=-\Delta_{\alpha}u_{j-1}
= -r^{-\alpha}(r^{\alpha}u'_{j-1})'$, $j=1,\dots,m$ and using the initial conditions $u_{j-1}(0)=a_{j-1}$
and $u'_{j-1}(0)=0$, we obtain
\begin{equation}\label{Q-Fubi}
    u_{j-1}(r)
= a_{j-1}
- \int_0^r t^{-\alpha}\int_0^t s^{\alpha}u_j(s)\,ds\,dt,\;\; j=1, \cdots, m.
\end{equation}
As in \cite[p. 71]{Walter}, by applying Fubini's theorem to \eqref{Q-Fubi}, we can see that \eqref{PVILocal} is equivalent to the system of Volterra integral equations
\begin{equation}\label{Volterra-re}
u_{j-1}(r)
= a_{j-1}-\int_0^r k(r,s)u_j(s)\,ds,
\end{equation}
where
\[
k(r,s)=s^{\alpha}\int_s^r t^{-\alpha}\,dt.
\]
Since $(s/t)^{\alpha}\le 1$ for $0\le s\le t$, we have
\begin{equation}\label{kernel-bound}
0\le k(r,s)\le r-s,
\quad 0\le s\le r\le \delta.
\end{equation}
Set $B_{d}=\{x\in \mathbb{R}^m\,:\, |x-a|\le d\}$, where $|x|=\max_{0\le j\le m-1} |x_j|$ for $x\in\mathbb{R}^m$. Now, let $X=C(I_{\delta}, B_d)$ be the space of continuous functions $\varphi: I_{\delta}\to B_{d}$ endowed with the metric
\[d(\varphi,\psi)
=
\max_{0\le j\le m-1} \|\varphi_j-\psi_j\|_{\infty},\;\;\mbox{with}\;\; \varphi=(\varphi_0,\cdots, \varphi_{m-1})\;\; \mbox{and}\;\; \psi=(\psi_0,\cdots, \psi_{m-1})
\]
where $\|\varphi\|_{\infty}=\sup_{I_\delta}|\varphi|$. Now, for each $\varphi\in X$, we set $T\varphi: I_{\delta}\to\mathbb{R}^m$ with $T\varphi=((T\varphi)_0, \cdots, (T\varphi)_{m-1})$, where 
\[
\left\{
\begin{aligned}
(T\varphi)_{m-1}(r)
&= a_{m-1}
- \int_0^r k(r,s)\,
|\varphi_0(s)|^{2^*-2}\varphi_0(s)\,ds,\\
(T\varphi)_{j-1}(r)
&= a_{j-1}
- \int_0^r k(r,s)\varphi_j(s)\,ds,
\qquad j=1,\dots,m-1.
\end{aligned}
\right.
\]
According with \eqref{Volterra-re}, the  fixed points of $T$ are precisely the solutions of
\eqref{PVILocal} on $I_\delta$.

From \eqref{kernel-bound}, for $j=1,\dots,m-1$ we obtain
\begin{equation}\label{Iva-1}
    |(T\varphi)_{j-1}(r)-a_{j-1}|
\le \int_0^r (r-s)|\varphi_j(s)|\,ds
\le \frac{\delta^2}{2}\big(d+|a|\big),
\end{equation}
where we have used that $\|\varphi_j\|_{\infty}\le \|\varphi_j-a_j\|_{\infty}+|a_j|\le d+|a|$. Further, 
\begin{equation}\label{Iva-2}
|(T\varphi)_{m-1}(r)-a_{m-1}|
\le \int_0^r (r-s)|\varphi_0(s)|^{2^*-1}\,ds \le \frac{\delta^2}{2}\big(d+|a|\big)^{2^*-1}.
\end{equation}
From \eqref{Iva-1} and \eqref{Iva-2}, we have that $T\varphi$ is continuous on $[0,\delta]$ with $(T\varphi)(0)=a$. In addition, from \eqref{delta-peq}, \eqref{Iva-1} and \eqref{Iva-2} we have
\begin{equation*}
    |(T\varphi)(r)-a|\le d,\;\;\mbox{for all}\;\; r\in I_{\delta}.
\end{equation*}
So, $T(X)\subset X$. Now, for $\varphi,\psi\in X$ and $j=1,\dots,m-1$,   we can write 
\begin{equation}\label{Contra-1}
    |(T\varphi)_{j-1}(r)-(T\psi)_{j-1}(r)|
\le \int_0^r (r-s)|\varphi_j(s)-\psi_j(s)|\,ds
\le \frac{\delta^2}{2}d(\varphi,\psi),\;\; \mbox{for}\;\; r\in I_{\delta}.
\end{equation}
We recall that for any  $p \ge 2$ it holds
$$\bigl| |x|^{p-2}x - |y|^{p-2}y \bigr| \le (p-1)(|x| + |y|)^{p-2} |x - y|,\;\;\; x, y\in \mathbb{R}.$$
Thus, since $\|\varphi_0\|_{\infty}, \|\psi_0\|_{\infty}\le d+|a|$, we obtain
\begin{equation} \nonumber
\begin{aligned}
 \big||\varphi_0|^{2^*-2}\varphi_0-|\psi_0|^{2^*-2}\psi_0\big|
&\le (2^*-1)\big((\|\varphi_0\|_{\infty}+\|\psi_0\|_{\infty}\big)^{2^*-2}\|\varphi_0-\psi_0\|_{\infty}   \\
& \le (2^*-1)\big(2(d+|a|)\big)^{2^*-2}d(\varphi,\psi),
\end{aligned}    
\end{equation}
on $I_\delta$. Then, 
\begin{equation} \label{Contra-2}
\begin{aligned}
 |(T\varphi)_{m-1}(r)-(T\psi)_{m-1}(r)|
& \le \frac{\delta^2}{2}(2^*-1)\big(2(d+|a|)\big)^{2^*-2}d(\varphi,\psi),\;\;\mbox{for}\;\; r\in I_{\delta}.
\end{aligned}    
\end{equation}
From \eqref{delta-peq}, \eqref{Contra-1} and \eqref{Contra-2}, we obtain
\begin{equation}
    d(T\varphi, T\psi)\le \frac{1}{2}d(\varphi,\psi).
\end{equation}
Hence, $T:X\to X$ is a contraction, and by the Banach contraction principle, $T$ admits a unique fixed point
$\varphi\in X$. 
\end{proof}

\begin{corollary}\label{v=w} Suppose \eqref{SSS-condi} holds. Then, for fixed $a=(a_0,\cdots, a_{m-1})\in \mathbb{R}^m$, the problem 
\begin{equation}\label{PVIGlobal}
\left\{\begin{aligned}
&\left.
\begin{aligned}
&u_0=u\\
&u_j=-\Delta_{\alpha}u_{j-1}, \quad j=1,\cdots,m-1\\
&-\Delta_{\alpha}u_{m-1}= |u|^{2^{*}-2}u\\
\end{aligned}
\right\}
\quad \text{in } (0,\infty),
\\
& u_j(0)=a_j\;\; \mbox{and}\;\;  u^{\prime}_j(0)=0,\quad j=0,1,\cdots,m-1
\end{aligned}\right.
\end{equation} 
admits at most one solution  $u\in C^{2m}(0,\infty)$.
\end{corollary}
\begin{proof}
   Suppose that $v,w\in C^{2m}(0,\infty)$  are solutions of \eqref{PVIGlobal}, with initial conditions understood in the sense of \eqref{l-IC}.  Let 
    $\eta=\sup\{\rho>0\;:\: v\equiv w\;\; \mbox{on}\;\; [0,\rho]\}$. By Theorem~\ref{thm-uniq} we have $0<\delta\le\eta\le \infty$. Assume by contradiction that $\eta<\infty$. From $v,w\in C^{2m}(0,\infty)$, for   $j=0,1,\cdots,m-1$ we can see from \cite[Lemma~2.9]{JN-JF} that both $v_j$ and $w_j$ are at least $C^1$ in a small open neighborhood of $\eta>0$. Since $v=w$ for $0\le r<\eta$, by continuity  $v_j(\eta)=w_j(\eta)$ and $v^{\prime}_j(\eta)=w^{\prime}_j(\eta)$. Thus,  both $v$ and $w$ solve the initial value problem
    \begin{equation}\label{PVISw-final}
\left\{\begin{aligned}
&\left.
\begin{aligned}
&u_0=u\\
&u_j=-\Delta_{\alpha}u_{j-1}, \quad j=1,\cdots,m-1\\
&-\Delta_{\alpha}u_{m-1}= |u|^{2^{*}-2}u\\
\end{aligned}
\right\}
\quad \text{in } (\eta,\infty),
\\
& u_j(\eta)=v_j(\eta)\;\; \mbox{and}\;\;  u^{\prime}_j(\eta)=v^{\prime}_j(\eta),\quad j=0,1,\cdots,m-1
\end{aligned}\right.
\end{equation} 
Since $\eta>0$, \eqref{PVISw-final} does not have singular term on $[\eta, \infty)$, then we can apply the classical Picard-Lindel\"{o}f uniqueness theorem to ensures that $v\equiv w$ on $[\eta, \eta+\epsilon)$ for some $\epsilon>0$ which contradicts the definition of $\eta$.
\end{proof}

\section*{Acknowledgement}
We are deeply grateful to the anonymous referee for observing the validity of Lemma~\ref{lemma-iterationRev}, which  allowed us to clarify the proof of Theorem~\ref{Teorema unicidade}.
\printbibliography

\end{document}